\documentclass[11pt,oneside]{amsart}
\usepackage{amsmath,amssymb,amsthm}
\usepackage[all]{xy}
\usepackage{mathrsfs}
\usepackage[includehead,margin=4cm]{geometry}
\swapnumbers 
\numberwithin{equation}{subsection}

\newtheorem{thm}[subsection]{Theorem}
\newtheorem*{thm*}{Theorem}
\newtheorem{lem}[subsection]{Lemma}
\newtheorem*{lem*}{Lemma}
\newtheorem{prop}[subsection]{Proposition}
\newtheorem*{prop*}{Proposition}
\newtheorem{cor}[subsection]{Corollary}
\newtheorem*{cor*}{Corollary}
\theoremstyle{definition}
\newtheorem{defn}[subsection]{Definition}
\newtheorem*{defn*}{Definition}
\newtheorem{example}[subsection]{Example}
\newtheorem*{example*}{Example}
\newtheorem{rmk}[subsection]{Remark}
\newtheorem*{rmk*}{Remark}

\newtheorem*{rmks*}{Remarks}

\newcommand{\N}{\mathbb{N}}
\newcommand{\Z}{\mathbb{Z}}
\newcommand{\Q}{\mathbb{Q}}
\newcommand{\R}{\mathbb{R}}
\newcommand{\C}{\mathbb{C}}
\newcommand{\X}{\mathbb{X}}
\newcommand{\A}{\mathcal{A}}
\newcommand{\AQ}{\mathbb{A}}
\newcommand{\BB}{\mathbf{B}}

\newcommand{\ad}{\operatorname{ad}}

\newcommand{\Hom}{\operatorname{Hom}}

\newcommand{\gl}{\mathfrak{gl}}
\newcommand{\g}{\mathfrak{g}}

\newcommand{\sqbinom}[2]{\genfrac{[}{]}{0pt}{}{#1}{#2}}
\newcommand{\UU}{\mathbf{U}}

\newcommand{\Schur}{\mathbf{S}}
\newcommand{\aschur}{{}_\A\Schur}

\newcommand{\bil}[2]{\langle #1,\;#2 \rangle}

\renewcommand{\le}{\leqslant}
\renewcommand{\ge}{\geqslant}
\newcommand{\domeq}{\unrhd}
\newcommand{\dom}{\rhd}
\newcommand{\ldomeq}{\unlhd}
\newcommand{\ldom}{\lhd}

\newcommand{\wt}{\operatorname{wt}}
\newcommand{\rk}{\operatorname{rank}}

\title{Cellular bases of generalized $q$-Schur algebras}

\author{Stephen Doty} \address{\footnotesize Loyola University Chicago,
  Chicago, Illinois 60626 U.S.A.}
\email{doty@math.luc.edu, tonyg@math.luc.edu}
\author{Anthony Giaquinto}

\begin{document}
\begin{abstract}\noindent 
  We show that cellular bases of generalized $q$-Schur algebras can be
  constructed by gluing arbitrary bases of the cell modules and their
  dual basis (with respect to the anti-involution giving the cell
  structure) along defining idempotents. For the rational form, over
  the field $\Q(v)$ of rational functions in an indeterminate $v$, our
  proof of this fact is self-contained and independent of the
  theory of quantum groups. In the general case, over a commutative
  ring $\Bbbk$ regarded as a $\Z[v,v^{-1}]$-algebra via specialization
  $v \mapsto q$ for some chosen invertible $q \in \Bbbk$, our argument
  depends on the existence of the canonical basis.
\end{abstract}
\maketitle

\section*{Introduction}\noindent
Most of this paper amounts to a self-contained study of the rational
form $\Schur(\pi)$ of a generalized $q$-Schur algebra by means of its
idempotent presentation (see \cite{DG:announce, DG:PSA,
  Doty:PGSA}). Our earlier papers studied these algebras via a
descent from the associated quantized enveloping algebra $\UU$. This
paper adopts the opposite viewpoint, taking the aforementioned
presentations as the starting point. Sections \ref{sec:prelim} --
\ref{sec:cellular} are self-contained; that is, independent of the
theory of quantum groups. We use only the given presentation of the
rational form and some basic facts (e.g., Weyl's theorem) about
semisimple Lie algebras (in the proof of Theorem
\ref{thm:simplicity}).  Our first main result is Theorem
\ref{thm:cellular}: assuming that bases of the cell modules have been
determined, a cellular basis (in the sense of Graham and Lehrer
\cite{GL}) of the algebra is obtained by gluing each cell module basis
and its dual basis (with respect to the defining anti-involution)
along the idempotent indexed by the highest weight. The cellular bases
so obtained are in general new, even in type A, and even for the
classical $q$-Schur algebras.

The second part of the paper, starting in
Section~\ref{sec:cellular-A}, is not self-contained.  In
\cite{Doty:PGSA} it was shown that the canonical basis of the divided
power $\A$-form ${}_\A\dot{\UU}$ of Lusztig's ``modified form'' of the
corresponding quantized enveloping algebra $\UU$ descends to a
cellular basis of the $\A$-form of the $q$-Schur algebra.  Here $\A =
\Z[v,v^{-1}]$. Starting in Section~\ref{sec:cellular-A} we assume the
results of \cite{Doty:PGSA} in order to know that the $\A$-form
${}_\A\Schur(\pi)$ of the algebra $\Schur(\pi)$ is free as an
$\A$-module, and in order to identify the kernels of various
homomorphisms to smaller $q$-Schur algebras. Our second main result is
Corollary \ref{cor:k-basis}, which extends the procedure of Theorem
\ref{thm:cellular} to the $\Bbbk$-form $\Schur_q(\pi)$ of
$\Schur(\pi)$, obtained by specializing $v$ to some chosen invertible
$q \in \Bbbk$. As an application of our approach, in
Section~\ref{sec:filtration} we give an explicit construction of
certain interesting filtrations of projective $\Schur_q(\pi)$-modules
coming from an idempotent decomposition of the regular representation.

Generalized Schur algebras $S(\pi)$ were
introduced by Donkin \cite{Donkin:SA1, Donkin:SA2, Donkin:SA3},
extending from type A to arbitrary type some of the results of Green's
monograph \cite{Green:book}. About the same time, the $q$-Schur
algebras in type A were introduced by Dipper and James \cite{DJ1, DJ2}
(see also Jimbo \cite{Jimbo}).
The generalized $q$-Schur algebras $\Schur_q(\pi)$ of this paper are
$q$-analogues of Donkin's generalized Schur algebras; they include the
classical $q$-Schur algebras in type A as a special case. In general,
$\Schur_q(\pi)$ depends on a root datum, a finite saturated set $\pi$
of dominant weights, and a chosen invertible element $q$ in the ground
ring $\Bbbk$. In this generality, their rational form $\Schur(\pi)$
first appeared \S29.2 of Lusztig's book \cite{Lusztig:book}, where
they are denoted by $\dot{\UU}/\dot{\UU}[P]$, where $P$ is the
complement of $\pi$ in the set of dominant weights. For other
approaches to generalized $q$-Schur algebras emerging around the same
time see \cite{Du-Scott} and \cite{Woodcock}. The presentation of the
rational form $\Schur(\pi)$ which forms the starting point for this
paper was obtained in \cite{Doty:PGSA} in general, extending earlier
results obtained in \cite{DG:PSA} for type A. See also \cite{DP, DGS1,
  DGS2, McGerty} for further results related to this point of view.

While our main results are obtained for the quantum case, the methods
also apply to the ``classical'' case; i.e., the original generalized
Schur algebras $S(\pi)$ as defined by Donkin.  The reader interested
only in that case should replace $\Q(v)$ by $\Q$ and $\A$ by $\Z$
throughout the exposition, making any other necessary adjustments as
needed. (One virtue of the defining presentation of $\Schur(\pi)$
used here is that it makes sense when $v=1$.)

We thank R.~Rouquier for pointing out that the quantum Serre relations
follow from the defining relations of $\Schur(\pi)$, and S.~Donkin for
pointing out a gap in an earlier version of the paper. Rouquier
directed our attention to Corollary 4.3.2 in \cite{Chriss-Ginzburg}
and Proposition B.1 in \cite{KMPY}. However, we do not need those
results --- in fact Section \ref{sec:Serre} of this paper contains a
new proof of them.

\section{Preliminaries}
\label{sec:prelim}\noindent
The definition of generalized quantized Schur algebras depends on a
chosen root datum and a saturated set $\pi$ of dominant weights. We
gather the needed background notions here.

\subsection{}\label{ss:Cartan-datum}
We fix a Cartan datum $(I, \cdot)$, consisting of a finite set $I$ and
an integer-valued symmetric bilinear form $(\nu, \nu') \mapsto \nu
\cdot \nu'$ on the free abelian group $\Z I$ generated by $I$,
satisfying:

(a) \quad $i \cdot i \in \{2,4,6, \dots\}$ for any $i \in I$; 

(b) \quad $2 \frac{i \cdot j}{i \cdot i} \in \{0, -1, -2, \dots\}$ for
any $i \ne j$ in $I$.

\noindent
We assume throughout this paper that the given Cartan datum is of
\emph{finite type}; i.e., the symmetric matrix $(i \cdot j)$ indexed
by $I \times I$ is positive definite. This is equivalent to the
requirement that the Weyl group $W$ (see \S2.1.1 in
\cite{Lusztig:book} for a precise definition of $W$ in this context)
associated to $(I, \cdot)$ is finite.

\subsection{}\label{ss:root-datum}
We assume given a \emph{root datum} $(\X,\Pi,\X^\vee,\Pi^\vee)$ of
type $(I,\cdot)$. This means that:

(a)\quad $\X$, $\X^\vee$ are finitely generated free abelian groups
connected by a bilinear pairing $\bil{\ }{\ } \colon \X^\vee \times \X
\to \Z$, such that the map $h \mapsto \bil{h}{\ }$ from $\X^\vee$ into
the dual $\Hom_\Z(\X, \Z)$ is an isomorphism.

(b)\quad $\Pi = \{ \alpha_i \colon i \in I\}$, $\Pi^\vee = \{
\alpha^\vee_i \colon i \in I\}$ are finite subsets of $\X$, $\X^\vee$
respectively. (They are called the sets of simple roots and coroots,
respectively.)

(c)\quad $\bil{\alpha^\vee_i}{\alpha_j} = 2\frac{i \cdot j}{i \cdot
  i}$ for all $i,j \in I$.

\medskip

These properties imply that the square matrix $(a_{ij})_{i,j \in
  I}$ defined by $a_{ij} = \bil{\alpha_i^\vee}{\alpha_j}$ is a
symmetrizable generalized Cartan matrix.

\subsection{}\label{d_i}
We put $d_i := \frac{i \cdot i}{2}$ for each $i \in I$. Then the
matrix $(d_i a_{ij})_{i,j \in I}$ is symmetric. The vector $(d_i)_{i
  \in I}$ is known as the symmetrizing vector; we have all $d_i = 1$
in the simply-laced case.

\subsection{}\label{ss:Lie-algebra}
The reductive finite dimensional Lie algebra $\widetilde{\g}$ attached
to the given root datum is the Lie algebra over $\Q$ generated by
elements $e_i$, $f_i$ ($i \in I$) and $h \in \X^\vee$ subject to the
relations:

(a)\quad $[h, h'] = 0$;

(b)\quad $[e_i,f_j] = \delta_{ij} \alpha_i^\vee$;

(c)\quad $[h,e_i] = \bil{h}{\alpha_i} e_i$,\quad $[h,f_i] =
-\bil{h}{\alpha_i} f_i$;

(d)\quad $(\operatorname{ad} e_i)^{1-a_{ij}}(e_j) = 0$, \quad
$(\operatorname{ad} f_i)^{1-a_{ij}}(f_j) = 0$ \quad (for $i \ne j$)

\noindent
holding for all $i,j \in I$ and all $h, h' \in \X^\vee$. This Lie
algebra will be needed only in the proof of Theorem
\ref{thm:simplicity}. (Actually, only its semisimple part $\g$, the
Lie subalgebra generated by all $e_i, f_i, h_i :=[e_i,
  f_i]=\alpha_i^\vee$ ($i \in I$) is needed there.

\subsection{}\label{ss:notation}
We will need the partial order $\domeq$ on $\X$ (the dominance order)
defined by $\lambda \domeq \mu$ if and only if $\lambda - \mu =
\sum_{i \in I} n_i \alpha_i$ where the $n_i \in \Z_{\ge 0}$ for all
$i$.  We write $\lambda \dom \mu$ if $\lambda \domeq \mu$ and $\lambda
\ne \mu$. We also write $\mu \ldomeq \lambda$ (respectively, $\mu
\ldom \lambda$) if $\lambda \domeq \mu$ (resp., $\lambda \dom \mu$).
Recall that the set $\X^+$ of dominant weights is defined by
\[
   \X^+ = \{ \lambda \in \X \colon \bil{\alpha_i^\vee}{\lambda} \ge 0,
   \text{ all } i \in I\}.
\]

\subsection{}\label{saturated}
We say that a set $\pi$ of dominant weights is \emph{saturated} if
$\pi$ contains all dominant predecessors of its elements; i.e., if
$\mu' \ldom \mu$ for $\mu' \in \X^+$ and $\mu \in \pi$ implies that
$\mu' \in \pi$.  For example, for any $\lambda \in \X^+$ the set
$\pi_\lambda = \X^+[\ldomeq \lambda] = \{\mu \in \X^+ \colon \mu
\ldomeq \lambda\}$ is saturated. In general, saturated subsets of
$\X^+$ are unions of such sets.

\subsection{}\label{ss:scalars}
Let $\Q(v)$ be the field of rational functions in an indeterminate
$v$.  For $a\in \Z$, $t\in \N$ we set
\[
   \sqbinom{a}{t} = \prod_{s=1}^t \frac{v^{a-s+1} - v^{-a+s-1}} {v^s -
   v^{-s}}.
\]
By 1.3.1(d) in \cite{Lusztig:book} this is an element of
the ring $\A = \Z[v,v^{-1}]$.  For any integer $n$ we set
\[
  [n] = \sqbinom{n}{1} = \frac{v^n - v^{-n}}{v - v^{-1}}
\]
and if $n \ge 0$ we set $[n]^! = [1]\cdots [n-1]\,[n]$ (with $[0]^!$
defined to be 1 as usual). Then it follows that
\[
\sqbinom{a}{t} = \frac{[a]^!}{[t]^!\,[a-t]^!}
\qquad \text{for all $0 \le t \le a$.}
\]

\subsection{}
We set $v_i = v^{d_i}$ for each $i \in I$, where $d_i = \frac{i\cdot
  i}{2}$ as in \ref{d_i}. More generally, given any rational function
$P \in \Q(v)$ we let $P_i$ denote the rational function obtained from
$P$ by replacing $v$ by $v_i$. In particular, this convention applies
to the notations $[n]_i$, $[n]^!_i$, and $\sqbinom{a}{t}_i$.

\section{The algebra $\Schur(\pi)$}\label{sec:Spi}\noindent
Henceforth we fix a root datum $(\X,\Pi,\X^\vee,\Pi^\vee)$ of type
$(I,\cdot)$. We always assume it is of finite type. We write $\Pi =
\{\alpha_i : i \in I\}$ and $\Pi^\vee = \{\alpha_i^\vee : i \in
I\}$. We now define the main object of study in this paper.

\subsection{Definition}\label{ss:defrels}
Let $\pi$ be a given finite saturated subset of the poset $\X^+$ of
dominant weights.  The generalized $q$-Schur algebra $\Schur(\pi)$
associated to $\pi$ is the associative algebra (with $1$) over $\Q(v)$
given by generators $E_i, F_i$ ($i \in I$) and $1_\lambda$ ($\lambda
\in W\pi$) subject to the relations:

(a)\quad $1_\lambda 1_\mu = \delta_{\lambda\mu} 1_\lambda, \qquad
\sum_{\lambda\in W\pi} 1_\lambda = 1$;

(b)\quad  $E_i F_j - F_j E_i = \delta_{ij} \sum_{\lambda\in W\pi}
[\bil{\alpha_i^\vee}{\lambda}]_i 1_\lambda$;

(c)\quad $E_i 1_\lambda = 1_{\lambda + \alpha_i} E_i$, $1_\lambda E_i
= 1_{\lambda-\alpha_i} E_i$, $F_i 1_\lambda = 1_{\lambda - \alpha_i}
F_i$, $1_\lambda F_i = F_i 1_{\lambda+\alpha_i}$

\noindent
for any $\lambda, \mu \in W\pi$, and $i,j \in I$ where in the right
hand side of relations (c) we stipulate that $1_{\lambda\pm\alpha_i} =
0$ whenever $\lambda \pm \alpha_i \notin W\pi$.

Note that it makes sense to formally set $v=1$ in the definition since
there are no denominators to cause any trouble. By doing so we get
another algebra over $\Q$, which will be denoted by $S(\pi)$. This is
the associative $\Q$-algebra defined by the same set of generators,
subject to the same relations (a)--(c), except in (b) the quantum
integer $[\bil{\alpha_i^\vee}{\lambda}]_i$ becomes the ordinary
integer $\bil{\alpha_i^\vee}{\lambda}$. We call $S(\pi)$ the classical
rational form of $\Schur(\pi)$. Later we will see that it coincides
with the rational form of a generalized Schur algebra in Donkin's
sense.

\subsection{Convention}
In the sequel we will consider expressions involving the symbol
$1_\mu$ for various $\mu \in \X$. We adopt the convention that in any
formulas involving the notation $1_\mu$ holding in some given
$\Schur(\pi)$, the symbol $1_\mu$ will be interpreted as 0 whenever
$\mu \notin W\pi$. 

In other words, we may when convenient regard $\Schur(\pi)$ as
generated by the family $1_\mu$ for all $\mu \in \X$ along with the
$E_i, F_i$ for $i \in I$, subject to the relations (a)--(c) of
\ref{ss:defrels} along with the additional relation
\[
   1_\mu = 0 \quad \text{ for any } \mu \in \X - W\pi.
\]

\subsection{}\label{ss:quotient-morphisms}
From the definition it is clear that for any saturated subset $\pi'
\subseteq \pi$ there is a natural surjective algebra homomorphism 
\begin{equation*}
  p_{\pi, \pi'} \colon \Schur(\pi) \to \Schur(\pi') 
\end{equation*}
sending generators $E_i \mapsto E_i$, $F_i \mapsto F_i$ and sending
$1_\lambda \mapsto 1_\lambda$ if $\lambda \in W\pi'$, $1_\lambda
\mapsto 0$ otherwise.

\subsection{}\label{ss:star-omega}
There is a unique $\Q(v)$-linear algebra anti-involution $*$ on
$\Schur(\pi)$ determined by the properties:

(a)\quad $(xy)^* = y^* x^*$ for all $x, y \in \Schur$,

(b)\quad $E_i^* = F_i$,\quad $F_i^* = E_i$,\quad $1_\mu^* = 1_\mu$
for any $i \in I$, $\mu \in \X$.  

\noindent
This is easily verified using the defining relations \ref{ss:defrels}.

Recall that the longest element $w_0$ of $W$ interchanges the positive
and negative roots, and thus $-w_0$ induces a permutation on the set
of simple roots. Given a saturated subset $\pi$ of $\X^+$ the set
$-w_0(\pi)$ is another saturated subset of $\X^+$. It follows
immediately from the defining relations \ref{ss:defrels} that there is
a unique algebra isomorphism $\omega = \omega_\pi \colon \Schur(\pi)
\to \Schur(-w_0(\pi))$ such that

(c)\quad  $\omega(E_i)=F_i$,\quad $\omega(F_i)=E_i$, \quad
$\omega(1_\mu)=1_{-\mu}$ for all $i \in I$, $\mu \in \X$. 

\noindent
Note that the inverse of $\omega_{\pi}$ is $\omega_{-w_0(\pi)}$.

\subsection{The $\A$-form  ${}_\A\Schur(\pi)$} \label{ss:divided-powers}
Until further notice we fix $\pi$ and write $\Schur$ for
$\Schur(\pi)$.  For any $i \in I$, $a \in \Z$ we introduce the
quantized divided powers
\[
E_i^{(a)} = \frac{E_i^a}{[a]_i^!},\qquad F_i^{(a)} =
\frac{F_i^a}{[a]_i^!} \qquad (a \ge 0).
\] 
If $a<0$ we define $E_i^{(a)} = F_i^{(a)} = 0$.  Let $\A =
\Z[v,v^{-1}]$. The $\A$-form ${}_\A\Schur = {}_\A\Schur(\pi)$ of
$\Schur(\pi)$ is the $\A$-subalgebra generated by all $E_i^{(a)}$,
$F_i^{(a)}$, and $1_\lambda$ (for $i \in I$, $a \ge 0$, $\lambda \in
W\pi$). 

\subsection{}\label{ss:A-quotient-morphisms}
For a saturated subset $\pi'$ of $\pi$, the restriction to ${}_\A
\Schur(\pi)$ of the homomorphism $p_{\pi, \pi'}$ defined in
\ref{ss:quotient-morphisms} gives a surjective algebra homomorphism
that we denote by ${}_\A p_{\pi,\pi'}$. It maps ${}_\A \Schur(\pi)$
onto ${}_\A \Schur(\pi')$.

The following consequences of the defining relations will be needed
later (compare with 23.1.3 in \cite{Lusztig:book}). Note that the sums
in parts (b) and (c) below are finite.

\begin{lem}\label{lem:commutation} 
 Let $\Schur = \Schur(\pi)$.  For any $a,b \ge 0$, $\mu \in
 W\pi$, the following identities hold in $\Schur$ and ${}_\A\Schur$:\rm

 (a)\quad $E_i^{(a)} 1_\mu = 1_{\mu+a\alpha_i} E_i^{(a)}$\,;\quad 
 $F_i^{(b)} 1_\mu = 1_{\mu-b\alpha_i} F_i^{(b)}$

 (b)\quad $E_i^{(a)} F_i^{(b)}1_\mu = \textstyle \sum_{t \ge 0}
 \sqbinom{a-b+\bil{\alpha_i^\vee}{\mu}}{t}_i \, F_i^{(b-t)} E_i^{(a-t)}
 1_\mu$

 (c)\quad $F_i^{(b)} E_i^{(a)}1_\mu = \textstyle \sum_{t \ge 0}
 \sqbinom{b-a-\bil{\alpha_i^\vee}{\mu}}{t}_i\, E_i^{(a-t)} F_i^{(b-t)}
 1_\mu$.
\end{lem}

\begin{proof}
Identity (a) is an obvious consequence of relation \ref{ss:defrels}(c)
and the definition of quantized divided powers in
\ref{ss:divided-powers}. One proves (b) by an elementary double
induction argument, left to the reader. Identity (c) can be proved by
a similar argument, or one may use the algebra isomorphism
$\omega_{\pi'}$ (see \ref{ss:star-omega}) taking $\Schur(\pi')$ onto
$\Schur(\pi)$, where $\pi' = -w_0(\pi)$.  By applying $\omega_{\pi'}$
to identity (b) for the algebra $\Schur(\pi')$, interchanging $a$ and
$b$, and replacing $-\mu$ by $\mu$, one obtains identity (c) for the
algebra $\Schur(\pi)$.
\end{proof}

\subsection{Plus, minus, and zero parts} \label{ss:plus-minus} 
Let $\Schur^+$ (respectively, $\Schur^-$) be the $\Q(v)$-subalgebra of
$\Schur$ generated by the $E_i$ (resp., the $F_i$) for $i \in I$.  Put
${}_\A\Schur^+ = {}_\A\Schur \cap \Schur^+$ and ${}_\A\Schur^- =
{}_\A\Schur \cap \Schur^-$. Clearly ${}_\A\Schur^+$ (respectively,
${}_\A\Schur^-$) is the $\A$-subalgebra of $\Schur$ generated by all
divided powers of the $E_i$ (resp., the $F_i$) for $i\in I$.

Let $\Schur^0$ be the $\Q(v)$-subalgebra of $\Schur$ generated by the
idempotents $1_\mu$ for $\mu \in W\pi$. Clearly we have the vector
space decomposition 
\[
  \textstyle \Schur^0 = \bigoplus_{\lambda \in W\pi} \Q(v)
1_\lambda. 
\]
Put ${}_\A\Schur^0 = {}_\A\Schur \cap \Schur^0$.  Then we have the
$\A$-module decomposition ${}_\A\Schur^0 = \bigoplus_{\lambda \in
  W\pi} \A 1_\lambda$.

The following result is an analogue for the algebras $\Schur(\pi)$ of
the triangular decomposition of a universal or quantized enveloping
algebra.

\begin{lem}[triangular decomposition] \label{lem:td}
  Let $\Schur = \Schur(\pi)$. We have $\Schur = \Schur^- \Schur^0
  \Schur^+$ and ${}_\A \Schur = {}_\A \Schur^- {}_\A \Schur^0 {}_\A
  \Schur^+$. The same equalities hold if the three factors on the
  right hand side are permuted in any order.
\end{lem}

\begin{proof}
  By the defining relations in \ref{ss:defrels} the algebra $\Schur$
  is spanned by elements of the form
\begin{equation}\label{eq:P}
  P = x_1 \cdots x_m 1_\mu
\end{equation}
where $x_1, \dots, x_m \in \{E_i, F_i \colon i \in I\}$ and $\mu \in
W\pi$. In any such product, we can use relation \ref{ss:defrels}(c) to
commute the idempotent $1_\mu$ to any desired position in the product
$P$, at the expense of replacing $\mu$ by some other weight
$\mu'$. For an element $P_0$ in the form \eqref{eq:P} let $d(P_0)$ be
the number of pairs $(j,j')$ such that $1 \le j \le j' \le m$ and $x_j
\in \{E_i \colon i \in I \}$, $x_{j'} \in \{F_i \colon i \in I \}$. We
claim that $P_0$ can be rewritten as a finite $\Q(v)$-linear
combination of elements of the form \eqref{eq:P} for which $d =
0$. This is proved by induction on $d$, using the following
consequences of the defining relations \ref{ss:defrels}(b):
\begin{equation}
 E_i F_j = F_j E_i\ (i \ne j); \quad
 E_i F_i 1_\mu = F_iE_i 1_\mu + [\bil{\alpha_i^\vee}{\mu}]_i 1_\mu .
\end{equation}
Thus it follows that all elements of $\Schur$ can be expressed as
linear combinations of products $P$ of the form \eqref{eq:P} with
$d(P)=0$. Since elements of the form \eqref{eq:P} with $d=0$ have all
occurrences of $F_j$ appearing before any $E_i$, it follows that
$\Schur = \Schur^- \Schur^+ \Schur^0$. To get the equality $\Schur =
\Schur^+ \Schur^- \Schur^0$ just repeat the argument with the $E$'s
and $F$'s interchanged. Finally, in any product of the form
\eqref{eq:P} with $d(P)=0$ we can use relation \ref{ss:defrels}(c) to
commute the idempotent to the left of all the $E_i$'s and the right of
all the $F_j$'s, so we obtain the equality $\Schur = \Schur^- \Schur^0
\Schur^+$. The other variations for $\Schur$ are obtained similarly.

To prove the claims for the $\A$-form ${}_\A\Schur$, one replaces
elements of the form \eqref{eq:P} with elements of the form 
\begin{equation}
  P = x_1 \cdots x_m 1_\mu
\end{equation}
where now $x_1, \dots, x_m \in \{E_i^{(b)}, F_i^{(b)} \colon i \in I,
b \ge 0\}$ and $\mu \in W\pi$. The argument then proceeds as above,
except that we must use Lemma \ref{lem:commutation} to do the necessary
rewriting. The details are left to the reader.
\end{proof}

\subsection{Weights and biweights}\label{ss:weights}
Let $\lambda, \mu \in W\pi$. We say that any $x \in \Schur$ is a
weight vector of left weight $\lambda$ if $1_\lambda x = x$, a weight
vector of right weight $\mu$ if $x1_\mu = x$, and a weight vector of
biweight $(\lambda, \mu)$ if $1_\lambda x 1_\mu = x$. Every weight
vector of biweight $(\lambda, \mu)$ has left weight $\lambda$ and
right weight $\mu$. Any product of generators in which at least one
factor is a generating idempotent $1_\nu$ is a weight vector of some
biweight.

The orthogonal idempotent decomposition of the identity given in
defining relation \ref{ss:defrels}(a) implies that if $M$ is any
subalgebra of $\Schur$ (resp., ${}_\A\Schur$) then 
\begin{equation}\label{eq:subspace-dec}
   \textstyle M = \bigoplus_{\lambda \in W\pi} 1_\lambda M =
   \bigoplus_{\mu \in W\pi} M1_\mu = \bigoplus_{\lambda,\mu \in
     W\pi} 1_\lambda M 1_\mu .
\end{equation}
In particular, this applies to $\Schur$ and ${}_\A\Schur$ themselves,
which therefore have the direct sum decompositions:
\begin{equation}\label{eq:biweight-S}
   \textstyle \Schur = \bigoplus_{\lambda, \mu \in W\pi} 1_\lambda \Schur
   1_{\mu}, \quad {}_\A\Schur = \bigoplus_{\lambda, \mu \in W\pi} 1_\lambda
   ({}_\A\Schur) 1_{\mu}.
\end{equation}
These decompositions are the biweight space decompositions of $\Schur$
and ${}_\A \Schur$.  Put $\Z \Pi = \sum \Z \alpha_i \subseteq \X$ (the
root lattice). Then we also have the decompositions:
\begin{equation}\label{eq:root-dec}
  \textstyle \Schur = \bigoplus_{\nu \in \Z \Pi} \Schur_\nu, \quad
  {}_\A\Schur = \bigoplus_{\nu \in \Z \Pi} {}_\A\Schur_\nu
\end{equation}
defined by the requirements: $\Schur_\nu\Schur_{\nu'} \subseteq
\Schur_{\nu+\nu'}$, $E_i \in \Schur_{\alpha_i}$, $F_i \in
\Schur_{-\alpha_i}$, and $1_\mu \in \Schur_{0}$.  Here ${}_\A
\Schur_\nu = {}_\A\Schur \cap \Schur_\nu$. Combining
\eqref{eq:root-dec} with the biweight space decompositions
\eqref{eq:biweight-S} we obtain the finer decompositions
\begin{equation}\label{eq:finer-dec}
  \textstyle \Schur = \bigoplus_{\nu \in \Z \Pi;\, \lambda, \mu \in
    W\pi} 1_\lambda \Schur_\nu 1_{\mu}, \quad {}_\A\Schur =
  \bigoplus_{\nu \in \Z \Pi;\, \lambda, \mu \in W\pi} 1_\lambda
  ({}_\A\Schur_\nu) 1_{\mu}.
\end{equation}
It is easy to see that in this decomposition $1_\lambda \Schur_\nu
1_{\mu} = 0$ (and thus $1_\lambda ({}_\A \Schur_\nu) 1_{\mu} = 0$)
unless $\lambda - \mu = \nu$.

Next we want to obtain a description of the algebras $\Schur^+$,
$\Schur^-$ and their $\A$-forms. 

\begin{lem}\label{lem:plus-minus}
  Let $\N\Pi = \{ \sum_{i \in I} n_i\alpha_i : 0 \le n_i \in \Z\}
  \subset \Z\Pi$.  We have gradings:
  
  (a)\quad $\Schur^+ = \bigoplus_{\nu \in \N\Pi}
  \Schur^+_\nu; \qquad \Schur^- = \bigoplus_{\nu \in \N\Pi}
  \Schur^-_{-\nu}$

  (b)\quad ${}_\A\Schur^+ = \bigoplus_{\nu \in \N\Pi}
  {}_\A\Schur^+_\nu; \qquad {}_\A\Schur^- = \bigoplus_{\nu \in \N\Pi}
  {}_\A\Schur^-_{-\nu}$

  \noindent
  defined by assigning $E_i$ to degree $\alpha_i$ and $F_i$ to degree
  $-\alpha_i$, for each $i \in I$. Each component $\Schur^+_\nu$
  (resp., $\Schur^-_{-\nu}$) is spanned over $\Q(v)$ by the set of all
  products $E_{j_1} \cdots E_{j_r}$ (resp., $F_{j_r} \cdots F_{j_1}$)
  such that the number of $j_k$ equal to $i$ is $n_i$. Each component
  ${}_\A\Schur^+_\nu$ (resp., ${}_\A\Schur^-_{-\nu}$) is equal to the
  intersection of $\Schur^+_\nu$ (resp., $\Schur^-_{-\nu}$) with
  ${}_\A\Schur$. 

  If there is no pair $\lambda, \mu \in W\pi$ such that $\lambda-\mu =
  \nu$ then $\Schur^+_\nu = 0 = \Schur^-_{-\nu}$. In particular,
  $\Schur^+$ and $\Schur^-$ are finite dimensional over $\Q(v)$, and
  ${}_\A\Schur^+$ and ${}_\A\Schur^-$ are finitely generated over
  $\A$.
\end{lem}

\begin{proof}
The statements about the graded decompositions (in the first
paragraph) are clear. For the last statement, let $x \in
\Schur^+_\nu$, $y \in \Schur^-_{-\nu}$. Then 
\begin{align*}
  x &= \textstyle (\sum_{\lambda} 1_\lambda) x (\sum_\mu 1_\mu) =
  \sum_{\lambda, \mu} 1_\lambda x 1_\mu = \sum_{\lambda, \mu}
  1_\lambda 1_{\mu+\nu}\, x = \delta_{\lambda,\, \mu+\nu}\,1_\lambda x \\
\intertext{and similarly} 
  y &= \textstyle (\sum_{\mu} 1_\mu) y (\sum_\lambda 1_\lambda) =
  \sum_{\lambda, \mu} 1_\mu y 1_\lambda = \sum_{\lambda, \mu} 1_\mu
  1_{\lambda-\nu}\, y = \delta_{\mu,\, \lambda-\nu}\,1_\mu y.
\end{align*}
This proves that $\Schur^+_\nu = 0 = \Schur^-_{-\nu}$ unless
$\lambda-\mu=\nu$ for some $\lambda, \mu \in W\pi$. In particular,
only finitely many of the summands in (a), (b) are non-zero. Since
each $\Schur^+_\nu$, $\Schur^-_{-\nu}$ is finite dimensional over
$\Q(v)$, and each ${}_\A\Schur^+_\nu$, ${}_\A\Schur^-_{-\nu}$ is
finitely generated over $\A$, the claims in the last paragraph follow.
\end{proof}

\begin{lem}\label{lem:fd}
  The algebra $\Schur = \Schur(\pi)$ is finite dimensional over
  $\Q(v)$ and its $\A$-form ${}_\A\Schur = {}_\A \Schur(\pi)$ is
  finitely generated as an $\A$-module. Furthermore, the generators
  $E_i, F_i$ ($i \in I$) are nilpotent in $\Schur$. 
\end{lem}

\begin{proof}
 The first claim follows from the preceding lemma and the triangular
 decomposition (Lemma \ref{lem:td}). The nilpotence of $E_i, F_i$
 follow from Lemma \ref{lem:commutation}(a), which implies that for
 any $\lambda \in W\pi$ we have $F_i^a 1_\lambda = E_i^a 1_\lambda =
 0$ for $a$ large enough. Then of course $F_i^a = F_i^a \cdot 1 =
 \sum_{\lambda \in W\pi} F_i^a 1_\lambda = 0$, and similarly for
 $E_i^a$.
\end{proof}

The following important observation will be needed later. 

\begin{lem}\label{lem:idem-straightening}
  Let $\Schur = \Schur(\pi)$, and suppose that $\lambda_0$ is a maximal
  element in $\pi$.  If $w = s_{i_r}s_{i_{r-1}} \cdots s_{i_1}$ is a
  reduced expression for $w \in W$ then we have
  \begin{equation*}
  1_{w(\lambda_0)} = F_{i_r}^{(a_r)} F_{i_{r-1}}^{(a_{r-1})} \cdots
  F_{i_1}^{(a_1)} 1_{\lambda_0} E_{i_1}^{(a_1)} \cdots
  E_{i_{r-1}}^{(a_{r-1})} E_{i_r}^{(a_r)}
  \end{equation*}
  where $a_j = \bil{\alpha_{i_j}^\vee}{s_{i_{j-1}} \cdots
    s_{i_1}(\lambda_0)}$, for each $j=1, \dots, r$.
\end{lem}

\begin{proof}
We proceed by induction on the length $r$ of $w \in W$. If $r=1$ then
$w = s_i$ is a simple reflection. Put $\mu = s_i(\lambda_0) = \lambda_0 -
\bil{\alpha_i^\vee}{\lambda_0} \alpha_i$ and set $a = b =
\bil{\alpha_i^\vee}{\lambda_0}$ in Lemma \ref{lem:commutation}(c). Then
we get
\[
   F_i^{(a)} E_i^{(a)} 1_{s_i(\lambda_0)} = 1_{s_i(\lambda_0)} 
\]
since $F_i 1_{s_i(\lambda_0)} = 1_{s_i(\lambda_0) - \alpha_i} F_i = 0$
because the weight $s_i(\lambda_0)$ is minimal in $W\pi$ in the
$\alpha_i$ direction. Applying Lemma \ref{lem:commutation}(a) we see
that $E_i^{(a)} 1_{s_i(\lambda_0)} = E_i^{(a)} 1_{\lambda_0-a \alpha_i} =
1_{\lambda_0} E_i^{(a)}$, so
\[
    F_i^{(a)} 1_{\lambda_0} E_i^{(a)}  = 1_{s_i(\lambda_0)}.
\]
This proves the result in case $r=1$.

Now assume that $r>1$ and write $w = s_{i_r} w'$ where $w' =
s_{i_{r-1}} \cdots s_{i_1}$. By induction the result holds for $w'$,
i.e.,
\[
1_{w'(\lambda_0)} =  F_{i_{r-1}}^{(a_{r-1})} \cdots
  F_{i_1}^{(a_1)} 1_{\lambda_0} E_{i_1}^{(a_1)} \cdots
  E_{i_{r-1}}^{(a_{r-1})} .
\]
For ease of notation set $i = i_r$. Put $\mu = w(\lambda_0)$ and set $a
= b = \bil{\alpha_i^\vee}{w'(\lambda_0)}$ in Lemma
\ref{lem:commutation}(c). Similar to the above, we get
\[
  F_i^{(a)} E_i^{(a)} 1_{w(\lambda_0)} = 1_{w(\lambda_0)} 
\]
since $w(\lambda_0)$ is minimal in the $\alpha_i$ direction. Again by
Lemma \ref{lem:commutation}(a) we have $E_i^{(a)} 1_{w(\lambda_0)} =
E_i^{(a)} 1_{w'(\lambda_0)-a \alpha_i} = 1_{w'(\lambda_0)} E_i^{(a)}$, so
\[
    F_i^{(a)} 1_{w'(\lambda_0)} E_i^{(a)}  = 1_{w(\lambda_0)} .
\]
Putting in $a = a_r$, $E_i = E_{i_r}$, and $F_i = F_{i_r}$ gives the
result.
\end{proof}

\begin{cor}\label{cor:kernel}
  Let $\lambda_0$ be maximal in $\pi$ and put $\pi' = \pi -
  \{\lambda_0\}$. 

  (a) The kernel of the surjective algebra quotient map $p_{\pi,\pi'}
  : \Schur(\pi) \to \Schur(\pi')$ defined in
  \ref{ss:quotient-morphisms} is equal to the two-sided ideal
  $\Schur(\pi) 1_{\lambda_0} \Schur(\pi) = \Schur(\pi)^- 1_{\lambda_0}
  \Schur(\pi)^+$.

  (b) The natural map $\Schur 1_{\lambda_0} \otimes_{\Q(v)}
  1_{\lambda_0} \Schur \to \Schur 1_{\lambda_0} \Schur$, given by $x
  \otimes y \mapsto xy$, is surjective.
\end{cor}

\begin{proof}
(a) We know by comparing the defining presentations of $\Schur(\pi)$
  and $\Schur(\pi')$ that the kernel of $p_{\pi,\pi'}$ is generated by
  the set of idempotents of the form $1_{w(\lambda_0)}$, for $w \in
  W$. The preceding lemma implies that all of these generators are
  contained in $\Schur(\pi) 1_{\lambda_0} \Schur(\pi)$. The final
  equality
\[
 \Schur(\pi) 1_{\lambda_0} \Schur(\pi) = \Schur(\pi)^- 1_{\lambda_0}
 \Schur(\pi)^+
\]
follows immediately from the triangular decomposition, Lemma
\ref{lem:td}.

(b) This is clear. 
\end{proof}

\section{Highest weight representations}\label{sec:reps}\noindent
We now embark on the study of the finite dimensional representations
of $\Schur = \Schur(\pi)$. As above, the finite saturated set $\pi$ is
fixed unless stated otherwise.

\subsection{}\label{ss:maxvecs}
Let $M$ be a left $\Schur$-module. We only consider finite dimensional
modules. Given $\lambda \in W\pi$, a vector $v \in M$ is a
\emph{weight vector} of weight $\lambda$ if $1_\lambda v =
v$. Evidently the subspace $1_\lambda M$ is the set of all weight
vectors of weight $\lambda$. We have the direct sum decomposition
\begin{equation}
  M  = \textstyle \bigoplus_{\lambda \in W\pi} 1_\lambda M.
\end{equation}
We call this the weight space decomposition of $M$ and call the
subspace $1_\lambda M$ a weight space of weight $\lambda$.  Note that
if we introduce, for any $h \in \X^\vee$, the element
\begin{equation}
  K_h = \textstyle \sum_{\lambda \in W\pi} v^{\bil{h}{\lambda}} 1_\lambda.
\end{equation}
then $K_h$ acts semisimply on any $\Schur$-module $M$, and acts as the
scalar $v^{\bil{h}{\lambda}}$ on the weight space $1_\lambda M$. This
justifies the terminology.

From relations \ref{ss:defrels}(c), (a) it follows that for any
$i \in I$ we have inclusions
\begin{equation}
  E_i (1_\lambda M) \subseteq 1_{\lambda+\alpha_i} M; \qquad F_i (1_\lambda M)
  \subseteq 1_{\lambda-\alpha_i} M
\end{equation}
for $\lambda\in W\pi$. So acting by $E_i$ on a weight vector increases the
weight by $\alpha_i$ and acting by $F_i$ decreases the weight by
$\alpha_i$.

\begin{defn}
Let $M$ be a given $\Schur$-module.  If $\lambda$ in $W\pi$ has the
property that $1_\lambda M \ne 0$ but $1_\mu M = 0$ for all $\mu \dom
\lambda$, then we say that $\lambda$ is a \emph{highest weight} of $M$
and we call any nonzero element of $1_\lambda M$ a \emph{highest
  weight vector}.  If $0 \ne x_0 \in M$ is a weight vector such that
$E_i \cdot x_0 = 0$ for every $i \in I$, then $x_0$ is called a
\emph{maximal vector} of $M$.
\end{defn}

A highest weight vector is also a maximal vector.  If $M \ne 0$ is a
finite dimensional $\Schur$-module then $M$ has at least one highest
weight vector, because there are only finitely many weights in $M$.
It follows that $M$ has at least one maximal vector.  We wish to study
the submodules of $M$ which are generated by a chosen maximal vector.

Let $I^*$ be the collection of all finite sequences $(i_1, i_2, \dots,
i_r)$ over the `alphabet' $I$, including the empty sequence
$\epsilon$.

\begin{lem}\label{lem:maxvec}
  Let $\Schur = \Schur(\pi)$.  Let $x_0$ be a maximal vector of weight
  $\lambda$ in a finite dimensional left $\Schur$-module $M$, and put
  $V = \Schur x_0$. Then:

  {\rm(a)} $V$ is the $\Q(v)$-span of elements of the form
  $F_{i_1}F_{i_2}\cdots F_{i_r} x_0$ for various finite sequences
  $(i_1, \dots, i_r)$ (including the empty sequence) in $I^*$.

  {\rm(b)} If $1_\mu V \ne 0$ then $\mu \ldomeq \lambda$, so $x_0$ is
  a highest weight vector of $V$.

  {\rm(c)} $\dim_{\Q(v)} 1_\lambda V = 1$. 

  {\rm(d)} $V$ is indecomposable, with a unique maximal submodule
    and a corresponding unique simple quotient.
\end{lem}

\begin{proof}
  Part (a) follows from the triangular decomposition $\Schur =
  \Schur^- \Schur^0 \Schur^+$, which implies that $V = \Schur^- x_0$
  since obviously $\Schur^+ x_0 = \Q(v)\cdot x_0$ and $\Schur^0 x_0 =
  \Q(v)\cdot x_0$. Part (b) follows from part (a), part (c) is
  obvious, and to get part (d), let $V'$ be the sum of all proper
  submodules and note that $V' \ne V$ since no proper submodule
  contains $x_0$.
\end{proof}

\begin{cor}
  Let $\Schur = \Schur(\pi)$. Any simple left $\Schur$-module has a
  unique highest weight vector $x_0$, up to scalar multiple.
\end{cor}

\begin{proof}
  Suppose that $M$ is a simple left $\Schur$-module. By the general
  theory of finite dimensional algebras over a field, $M$ is finite
  dimensional, since it appears as the top composition factor of some
  direct summand of the regular representation. Thus the simple
  $\Schur$-module $M$ must have a highest weight vector, which must
  also be a maximal vector.  Any maximal vector necessarily generates
  $M$ since $M$ is simple. Suppose that there are two highest weight
  vectors, respectively of weight $\lambda$ and $\lambda'$.  Then by
  part (b) of Lemma \ref{lem:maxvec} we must have $\lambda' \ldomeq
  \lambda$ and $\lambda \ldomeq \lambda'$, so $\lambda =
  \lambda'$. Then part (c) of Lemma \ref{lem:maxvec} forces the two
  highest weight vectors to be proportional, and the proof is
  complete.
\end{proof}

As another corollary of Lemma \ref{lem:maxvec} we show that highest
weight simple modules are uniquely determined by their highest weight.
This is important for the classification of the simple
$\Schur$-modules, which will be completed later.

\begin{cor}
  Let $\Schur = \Schur(\pi)$.  Let $L, L'$ be two
  simple left $\Schur$-modules, each of highest weight $\lambda$, for
  some $\lambda \in \pi$. Then $L$ is isomorphic to $L'$.
\end{cor}

\begin{proof}
(Compare with the proof of Theorem A of \S20.3 in \cite{Humphreys}.)
  Let $M = L \oplus L'$. Suppose that $y_0$, $y'_0$ are maximal
  vectors in $L$, $L'$ respectively. The left weight of both $y_0$ and
  $y'_0$ is $\lambda$. Put $x_0 = (y_0,y'_0)$. Then $x_0$ is a maximal
  vector in $M$, of left weight $\lambda$. Let $N$ be the submodule of
  $M$ generated by $x_0$. Lemma \ref{lem:maxvec} implies that $N$ has
  a unique simple quotient. But the natural projections $N \to L$, $N
  \to L'$ are $\Schur$-module epimorphisms, so $L \simeq L'$, as
  desired.
\end{proof}

The algebra $\Schur = \Schur(\pi)$ has a family of highest
weight modules, one a left module and one a right module, for each
$\lambda \in \pi$, which will be denoted by $\Delta(\lambda)$ and
$\Delta(\lambda)^*$, respectively.

\begin{prop}
  Put $\Schur = \Schur(\pi)$. For each $\lambda \in \pi$ there is a
  unique left $\Schur$-module $\Delta(\lambda)$ of highest weight
  $\lambda$ and a unique right $\Schur$-module $\Delta(\lambda)^*$ of
  highest weight $\lambda$ such that $\Delta(\lambda)$ is the left
  ideal $\Schur(\pi_\lambda) 1_\lambda$ and $\Delta(\lambda)^*$ the
  right ideal $\Schur(\pi_\lambda) 1_\lambda$ in the algebra
  $\Schur(\pi_\lambda)$, where $\pi_\lambda = \{\mu \in X^+ : \mu
  \ldomeq \lambda \}$.
\end{prop}

\begin{proof}
Evidently $\lambda$ is a maximal element of $\pi_\lambda$. The left
ideal $\Schur(\pi_\lambda)1_\lambda$ becomes a left $\Schur$-module
via the surjective algebra homomorphism
\[
  p_{\pi,\pi_\lambda} : \Schur = \Schur(\pi) \to \Schur(\pi_\lambda)
\] 
induced by the inclusion of the saturated set $\pi_\lambda \subseteq
\pi$ (see \ref{ss:quotient-morphisms}). Similarly for the right ideal
$1_\lambda \Schur(\pi_\lambda)$.
\end{proof}

Later we will see that the modules $\Delta(\lambda)$ are simple as
$\Schur$-modules, for all $\lambda \in \pi$, and that $\Schur$ is a
semisimple algebra.

\section{Semisimplicity in the rank 1 case}\label{sec:rank1}\noindent
We turn now to the case where the root datum (see \ref{ss:root-datum})
has rank 1. So we assume that $I = \{i\}$ is a singleton and $\X =
\X^\vee = \Z$, with $\alpha_i = 2$ and $\alpha_i^\vee = 1$. Dominant
weights are just nonnegative integers, and a set $\pi$ of nonnegative
integers is saturated if and only if $n \in \pi$ and $n-2 \ge 0$
implies that $n-2 \in \pi$. Any subset of $2\Z$ of the form $2\Z \cap
[0,x]$ is saturated, as is any subset of $2\Z+1$ of the form
$(2\Z+1)\cap [0,x]$, where $x$ is a positive real number and $[0,x]$
is the closed interval in the real line $\R$ with the given
endpoints. In general, a subset of $\X^+ = \Z_{\ge 0}$ is saturated if
and only if it is a union of such subsets.

We wish to classify the simple representations of $\Schur =
\Schur(\pi)$ in the rank 1 case. This prepares the way for the
classification of simple representations in higher ranks, and is also
needed in the derivation of the quantized Serre relations in Section
\ref{sec:Serre}.

\begin{lem}\label{lem:r1simples}
For $n \in \pi$ let $x_0$ be a highest weight vector of $\Delta(n)$
and $\Delta(n)^*$. Put $x_t = F_i^{(t)} \cdot x_0 \in \Delta(n)$ and
$x'_t = x_0 \cdot E_i^{(t)} \in \Delta(n)^*$, for $t \ge 0$. For $t<0$
put $x_t = x'_t = 0$. Then $x_{t}=0=x^\prime_{t}$ for all $t > n$, and
furthermore
\begin{itemize}
\item[(a)]\quad  $F_i \cdot x_t = [t+1]_i\, x_{t+1}$, \quad $E_i \cdot x_t =
  [n-(t-1)]_i \, x_{t-1}$

\item[(b)]\quad $x^\prime_t \cdot E_i = [t+1]_i\, x^\prime_{t+1}$, \quad
  $x^\prime_t \cdot F_i = [n-(t-1)]_i \, x^\prime_{t-1}$ 
\end{itemize}

\noindent
for $t = 0, 1, \dots, n$.  Thus $\{x_0, \dots, x_n\}$ is a
basis for $\Delta(n)$ and $\{x^\prime_0, \dots,
x^\prime_n\}$ is a basis for $\Delta(n)^*$, and both
$\Delta(n)$ and $\Delta(n)^*$ are simple $\Schur$-modules.
\end{lem}

\begin{proof} 
We know that $x_0$ is a nonzero scalar multiple of $1_n$, so we may as
well assume that $x_0 = 1_n$. We have $F_i \cdot F_i^{(t)} 1_n =
[t+1]_i F_i^{(t+1)} 1_n$ for all $t \ge 0$ by definition of the
quantized divided powers.  This gives the first formula in (a). If $t
> n$ then $F_i^{(t)} 1_n = 1_{n-2t} F_i^{(t)} = 0$ in $\Delta(n)$ by
Lemma \ref{lem:commutation}(a).  From Lemma \ref{lem:commutation}) we
have
\begin{equation*}
  E_i \cdot F_i^{(t)} 1_n = F_i^{(t)} E_i 1_n + [n-(t-1)]_i
  F_i^{(t-1)} 1_n 
\end{equation*}
in $\Delta(n)$, for all $t \ge 0$.  This implies the second formula in
(a) since $E_i 1_n = 0$ in $\Delta(n)$. The elements $x_0, \dots, x_m$
have distinct left weights, so they are linearly independent. They
clearly span $\Delta(n)$, and hence form a basis of $\Delta(n)$. The
formulas in (a) then imply that $\Delta(n)$ is generated, as an
$\Schur$-module, by any one of its basis elements $x_t$, for any $0
\le t \le m$. Thus $\Delta(n)$ is a simple module.

The formulas in (b) can be proved by applying the anti-involution $*$
to the formulas in (a), and all the rest of the claims about
$\Delta(n)^*$ follow.
\end{proof}

We can now prove our main result of this section, which classifies the
simple modules and shows that the algebra $\Schur$ is semisimple.

\begin{prop}\label{prop:r1ss}
  In the rank one case, the algebra $\Schur = \Schur(\pi)$ is
  semisimple, and a complete set of isomorphism classes of simple left
  $\Schur$-modules is given by $\{ \Delta(n) \colon n \in \pi \}$.
\end{prop}

\begin{proof}
We proceed by induction on the cardinality of $\pi$. If $\pi =
\{n_0\}$ then we have $\Schur = \Schur(1_{n_0} + 1_{-n_0})\Schur =
\Schur 1_{n_0}\Schur$, since $\Schur 1_{-n_0} \Schur \subseteq \Schur
1_{n_0}\Schur$ by Lemma \ref{lem:idem-straightening}. Hence by the
surjectivity in Corollary \ref{cor:kernel}(b) we have
\[
  \dim_{\Q(v)} \Schur \le (\dim_{\Q(v)} \Delta(n_0))^2. 
\]
On the other hand, the standard theory of finite dimensional algebras
implies that
\[
  \dim_{\Q(v)} \Schur \ge (\dim_{\Q(v)} \Delta(n))^2
\]
since $\Delta(n_0)$ is a simple $\Schur$-module.  This proves equality
of dimensions, and the claims follow in case $|\pi| = 1$.

Now assume that $|\pi| \ge 2$, and that $n_0$ is maximal in $\pi$. Put
$\pi' = \pi - \{n_0\}$. By induction the algebra $\Schur(\pi')$ is
semisimple and $\dim_{\Q(v)} \Schur(\pi') = \sum_{n \in \pi'}
(\dim_{\Q(v)} \Delta(n))^2$. By Corollary \ref{cor:kernel}(a) the kernel
of $p_{\pi, \pi'}$ equals the ideal $\Schur 1_{n_0} \Schur$, so by the
Corollary \ref{cor:kernel}(b) we have $\dim_{\Q(v)} \ker p_{\pi,\pi'} \le
(\dim_{\Q(v)} \Delta(n_0))^2$. Hence we have
\begin{align*}
  \dim_{\Q(v)}\Schur &= \dim_{\Q(v)} \ker p_{\pi,\pi'} +
  \textstyle\sum_{n \in \pi'} (\dim_{\Q(v)} \Delta(n))^2 \\ & \le
  \textstyle\sum_{n \in \pi} (\dim_{\Q(v)} \Delta(n))^2.
\end{align*} 
On the other hand, since the $\Delta(n)$ are pairwise non-isomorphic
simple modules (indeed, no two of them have the same highest weight),
the standard theory of finite dimensional algebras implies that
\[
  \dim_{\Q(v)} \Schur \ge \textstyle\sum_{n \in \pi} (\dim_{\Q(v)}
  \Delta(n))^2.
\]
It follows that the dimensions agree, and thus that $\Schur$ is
semisimple, with simple modules as stated.
\end{proof}

\begin{rmk}
The proof also establishes the fact that the natural multiplication
map
\[
\Delta(n_0)\otimes \Delta(n_0)^* \to \Schur 1_{n_0} \Schur
\] 
in Corollary \ref{cor:kernel}(b) is actually an isomorphism of
$\Schur$-bimodules (in rank 1).  This can be proved another way, by
observing that the spanning set of nonzero $F^{(b)} 1_{n_0} E^{(a)}$ in
$\Schur 1_{n_0} \Schur = \Schur^- 1_{n_0} \Schur^+$ is linearly
independent, since the elements have distinct biweights.
\end{rmk}

\section{Quantum Serre relations}\label{sec:Serre}\noindent
Whenever the rank $|I|$ of the underlying root datum is at least 2,
certain relations among the negative part generators $F_i$ ($i \in I$)
must hold in any $\Schur(\pi)$, and exactly the same relations must
hold among the positive part generators $E_i$ ($i \in I$). The
relations are known as the quantum Serre relations, since they are
quantum analogues of the usual Serre relations in the universal
enveloping algebra of a semisimple Lie algebra. They arise in the
present context as consequences of the defining relations for
$\Schur(\pi)$.

\subsection{}\label{ss:Si}
Throughout this section we assume $|I| \ge 2$ and fix a finite
saturated set $\pi \subset \X^+$, and set $\Schur = \Schur(\pi)$.  Fix
some $i \in I$. Let $\pi^{(i)}$ be the set of all integers
$\bil{\alpha_i^\vee}{\lambda} \ge 0$ such that $\lambda \in W\pi$.  If
$0\le n = \bil{\alpha_i^\vee}{\lambda}$ is in $\pi^{(i)}$ then so are
all nonnegative integers in the set $\{n, n-2, \dots, -(n-2),
-n\}$. Thus $\pi^{(i)}$ is a saturated set of nonnegative integers for
a rank 1 root datum as considered in the preceding section. The Weyl
group in this rank 1 situation may be regarded as the multiplicative
group $\{-1,1\}$ acting by multiplication and we write $\pm \pi^{(i)}$
short for $-\pi^{(i)} \cup \pi^{(i)}$.  For any $n \in \pm \pi^{(i)}$
we define
\begin{equation}
  1_n = \textstyle\sum_{\lambda \in W\pi \colon\,
    \bil{\alpha_i^\vee}{\lambda} = n} 1_\lambda
\end{equation}
regarded as an element of $\Schur=\Schur(\pi)$. Let $\Schur^{(i)}$ be
the subalgebra of $\Schur$ generated by $E_i$, $F_i$, and the $1_n$
($n \in \pm \pi^{(i)}$).

\begin{prop}\label{prop:Si}
  The subalgebra $\Schur^{(i)}$ is isomorphic to the associative
  algebra with $1$ given by the generators $E_i$, $F_i$ and $\{1_n
  \colon n \in \pm\pi^{(i)} \}$ and subject to the relations
\begin{enumerate}
\item[\rm(a)] \quad $1_m 1_n = \delta_{m,n} 1_n$ for all $m,n \in \pm
\pi^{(i)}$; \quad $\sum_{n \in \pm \pi^{(i)}} 1_n = 1$; 

\item[\rm(b)] \quad $E_iF_i - F_i E_i = \sum_{n \in \pm \pi^{(i)}}
  [n]_i 1_n$;

\item[\rm(c)] \quad $E_i 1_n = 1_{n+2} E_i$, $1_n E_i = 1_{n-2}
  E_i$, $F_i 1_n = 1_{n-2} F_i$, $1_n F_i = F_i 1_{n+2}$
\end{enumerate}
where in the right hand side of {\rm(c)} we interpret $1_{n\pm 2} = 0$
whenever $n\pm 2 \notin \pm\pi^{(i)}$. Thus, $\Schur^{(i)}$ is isomorphic
to a rank $1$ algebra determined by the saturated set $\pi^{(i)}$ and
a rank $1$ root datum.
\end{prop}

\begin{proof}
  Checking that relations (a)--(c) hold in $\Schur^{(i)}$ is trivial
  from the definitions and left to the reader. The fact that these
  relations hold in $\Schur^{(i)}$ implies that $\Schur^{(i)}$ is a
  homomorphic image of the rank 1 algebra defined by the set
  $\pi^{(i)}$ and a rank $1$ root datum. It follows from Proposition
  \ref{prop:r1ss} that $\Schur^{(i)}$ is semisimple.  To finish the
  argument, we need to compare dimensions. For this it is enough to
  show that $\Schur^{(i)}$ has at least as many simple modules as its
  rank 1 covering algebra.

  Let $n \in \pi^{(i)}$. By definition of $\pi^{(i)}$, there is some
  $\lambda \in W\pi$ such that $n = \bil{\alpha_i^\vee}{\lambda}$. If
  $\lambda \in \pi$ then we may regard $\Delta(\lambda)$ as an
  $\Schur^{(i)}$-module by restriction. The generating maximal vector
  $x_0 \in \Delta(\lambda)$ is still maximal for the action of
  $\Schur^{(i)}$, and hence the $\Schur^{(i)}$-submodule it generates
  is a simple module of highest weight $n$. If $\lambda \notin \pi$
  then there is some $\lambda' \in \pi$ and some $1 \ne w \in W$ such
  that $\lambda = w(\lambda')$. We may regard the ideal $J =
  \Schur(\pi_{\lambda'}) 1_{\lambda'} \Schur(\pi_{\lambda'})$ as an
  $\Schur$-module via the natural quotient map $\Schur \to
  \Schur(\pi_{\lambda'})$. The weight $\lambda$ is by assumption on
  the positive side of the reflecting hyperplane for the simple root
  $\alpha_i$, so $\lambda + \alpha_i \notin W\pi$ and thus the vector
  $1_\lambda \in J$ is killed by the action of $E_i$. This implies
  that the $\Schur^{(i)}$-submodule of $J$ generated by $1_\lambda$ is
  a simple module of highest weight $n$.
\end{proof}

\subsection{} \label{ss:K_h} We will need to consider the elements
$K_h = \textstyle \sum_{\lambda \in W\pi} v^{\bil{h}{\lambda}} \, 1_\lambda$ (any
$h \in \X^\vee$) defined in \ref{ss:maxvecs}(b).  Observe that $K_0 =
1$ and $K_h K_{h'} = K_{h+h'}$ for any $h,h' \in \X^\vee$. Consequently
$K_h$ is invertible and $K_h^{-1} = K_{-h}$. Since $K_h 1_\lambda =
v^{\bil{h}{\lambda}}\, 1_\lambda$ for any $\lambda \in W\pi$ it is
clear that the $K_h$ act semisimply on any finite dimensional
$\Schur$-module.

We are particularly interested in the elements $\overline{K}_i =
K_{d_i \alpha_i^\vee}$, $\overline{K}^{-1}_i = K_{-d_i \alpha_i^\vee}$
for our fixed $i \in I$. (See \ref{ss:root-datum} for the definition of
$d_i$.) We have
\begin{equation}
  \overline{K}_i = \textstyle \sum_{\lambda \in W \pi}
  v_i^{\bil{\alpha_i^\vee}{\lambda}} \, 1_\lambda \, ; \qquad
  \overline{K}^{-1}_i = \textstyle \sum_{\lambda \in W \pi}
  v_i^{-\bil{\alpha_i^\vee}{\lambda}} \, 1_\lambda
\end{equation}
and from the defining relations \ref{ss:defrels}(b) and the definition
of $[n]_i$ in \ref{ss:scalars} we obtain the relations
\begin{gather}
  E_i F_j - F_j E_i = \delta_{i,j} \frac{\overline{K}_i -
    \overline{K}^{-1}_i}{v_i - v_i^{-1}}\\  \overline{K}_i E_j
  \overline{K}_i^{-1} = v_i^{\bil{\alpha_i^\vee}{\alpha_j}} E_j\, ;
  \quad \overline{K}_i F_j \overline{K}_i^{-1} =
  v_i^{-\bil{\alpha_i^\vee}{\alpha_j}} F_j
\end{gather}
which hold in the algebra $\Schur = \Schur(\pi)$, for any $j \in I$.
From (a) and the definition of the $1_n$ we also have
\begin{equation}
  \overline{K}_i = \sum_{n \in \pm \pi^{(i)}} v_i^n\, 1_n; \qquad 
  \overline{K}_i^{-1} = \sum_{n \in \pm \pi^{(i)}} v_i^{-n}\, 1_n;
\end{equation}
this shows that $\Schur^{(i)}$ is generated by the four elements
$E_i$, $F_i$, $\overline{K}_i$, $\overline{K}_i^{-1}$.  A simple
direct calculation shows that for each $n \in \pi^{(i)}$ we have
\begin{equation}
  1_n = \prod_{m \in \pm \pi^{(i)} \colon\, m \ne n}
  \frac{\overline{K}_i - v_i^m}{v_i^n - v_i^m}.
\end{equation}
This is needed in the proof of the following alternative presentation
of the algebra $\Schur^{(i)}$.

\begin{prop}\label{prop:alt-Si}
  The algebra $\Schur^{(i)}$ is isomorphic to the associative algebra
  with $1$ given by the generators $E_i$, $F_i$, $\overline{K}_i$,
  $\overline{K}_i^{-1}$ and subject to the relations
  
  {\rm(a)} \quad $\overline{K}_i \overline{K}_i^{-1} = 1$; \quad $E_i F_i
    - F_i E_i = \dfrac{\overline{K}_i - \overline{K}^{-1}_i}{v_i -
      v_i^{-1}}$;

  {\rm(b)} \quad $\overline{K}_i E_i \overline{K}_i^{-1} = v_i^{2} E_i$;
    \quad $\overline{K}_i F_i \overline{K}_i^{-1} = v_i^{-2} F_i$;

  {\rm(c)} \quad $\prod_{n \in \pm \pi^{(i)}} \,(\overline{K}_i -
  v_i^n) = 0$.
\end{prop}

\begin{proof}
  Let $\Schur^{(i)}$ be the algebra given by the presentation in
  Proposition \ref{prop:Si}, and let $Z$ be the algebra given by the
  presentation of the current proposition. Define an algebra map
  $\psi$ from $\Schur^{(i)} \to Z$ by sending $E_i$ to $E_i$, $F_i$ to
  $F_i$, and $1_n$ to $\prod_{m \in \pm \pi^{(i)} \colon\, m \ne n}
  \frac{\overline{K}_i - v_i^m}{v_i^n - v_i^m}$. Then check that the
  elements $\psi(E_i)$, $\psi(F_i)$, and $\psi(1_n)$ for $n \in \pm
  \pi^{(i)}$ satisfy the defining relations (a)--(c) in
  \ref{prop:Si}. 

  On the other hand, define an algebra map $\psi'$ from $Z \to
  \Schur^{(i)}$ by sending $E_i$ to $E_i$, $F_i$ to $F_i$, 
  $\overline{K}_i$ to $\sum_{n \in \pm \pi^{(i)}} v_i^n\, 1_n$, and
  $\overline{K}_i^{-1}$ to $\sum_{n \in \pm \pi^{(i)}} v_i^{-n}\,
  1_n$. Verify that the
  elements $\psi'(E_i)$, $\psi'(F_i)$, $\psi'(\overline{K}_i)$,
  $\psi'(\overline{K}_i^{-1})$ satisfy the defining relations (a)--(c)
  given above. The result follows.
\end{proof}

\subsection{} \label{ss:ad-endo} From \ref{ss:K_h}(d) we have
$\overline{K}_i 1_n = v_i^n 1_n$ for each $n \in \pm \pi^{(i)}$. From
this it follows easily that if $M$ is an arbitrary
$\Schur^{(i)}$-module then
\begin{equation}
  1_n M = \{ x \in M \colon \overline{K}_i x = v_i^{n}\, x \}. 
\end{equation}
Thus, we can define weight vectors of weight $n$ in an
$\Schur^{(i)}$-module $M$ as elements $x \in M$ which are fixed under
left multiplication by $1_n$, or equivalently as eigenvectors for the
operator $\overline{K}_i$ belonging to the eigenvalue $v_i^n$.

Now we define $\Q(v)$-linear endomorphisms $\ad \overline{K}_i$, $\ad
\overline{K}^{-1}_i$, $\ad E_i$, and $\ad F_i$ of the algebra $\Schur
= \Schur(\pi)$ by the rules:
\begin{gather}
  (\ad \overline{K}_i) x = \overline{K}_i x \overline{K}_i^{-1} \, ;\quad
  (\ad \overline{K}^{-1}_i) x = \overline{K}^{-1}_i x \overline{K}_i \\
  (\ad E_i) x = E_i x - (\overline{K}_i x \overline{K}_i^{-1})E_i\, ; \quad
  (\ad F_i) x = (F_i x - x F_i) \overline{K}_i.
\end{gather}
for any $x \in \Schur$, where the products on the right hand side of
each equation take place in the algebra $\Schur$.  We claim that these
endomorphisms satisfy the defining relations of $\Schur^{(i)}$ given
in Proposition \ref{prop:alt-Si}. This is verified by direct
calculation, left to the reader. It follows that these endomorphisms
define an $\Schur^{(i)}$-module structure on the algebra $\Schur =
\Schur(\pi)$.

\subsection{}\label{ss:derive}
Now we can derive the quantum Serre relations. Taking $x = E_j$ for $j
\ne i$, we have from \ref{ss:ad-endo}(b), (c) and the relations
\ref{ss:K_h}(b), (c) that
\begin{equation}
  (\ad \overline{K}_i) E_j = v_i^{\bil{\alpha_i^\vee}{\alpha_j}} E_j\,
  ; \quad (\ad F_i) E_j = 0.
\end{equation}
From this it follows that $E_j$ is a \emph{lowest} weight vector of
weight $a_{ij} = \bil{\alpha_i^\vee}{\alpha_j}$ for the adjoint action
of $\Schur^{(i)}$. 

The classification of rank 1 representations worked out in Section
\ref{sec:rank1} tells us that the $\Schur^{(i)}$-submodule of $\Schur$
generated by $E_j$ is isomorphic to $\Delta(-a_{ij})$. In particular,
$(\ad E_i)^{-a_{ij}} E_j$ is a highest weight vector in the module,
and thus
\begin{equation}
  (\ad E_i)^{1-a_{ij}} E_j = 0.
\end{equation}
On the other hand, one can show by an easy induction on $r$ that for
any $r$ we have
\begin{equation}
  (\ad E_i)^{r} E_j = \sum_{s=0}^r (-1)^s \sqbinom{r}{s}_i E_i^{r-s}
  E_j E_i^s.
\end{equation}
This leads to the following result.

\begin{thm}\label{thm:q-Serre}
  Let the root datum be of rank at least $2$.  Let $\pi$ be an
  arbitrary finite saturated subset of $\X^+$ and let $\Schur =
  \Schur(\pi)$. For any $i,j \in I$ with $i \ne j$ the quantum
  Serre relations
  \begin{align}
    \textstyle\sum_{s=0}^{1-a_{ij}} (-1)^s \sqbinom{1-a_{ij}}{s}_i
    E_i^{1-a_{ij}-s} E_j E_i^s &= 0 ; \\ \textstyle \sum_{s=0}^{1-a_{ij}}
    (-1)^s \sqbinom{1-a_{ij}}{s}_i F_i^{1-a_{ij}-s} F_j F_i^s &= 0
  \end{align}
  hold in the algebra $\Schur$, where $a_{ij} =
  \bil{\alpha_i^\vee}{\alpha_j}$.
\end{thm}

\begin{proof}
  To get (a) we put $r = 1-a_{ij}$ and combine \ref{ss:derive}(b) and
  (c). To get (b), we let $\pi' = -w_0(\pi)$ and consider the algebra
  isomorphism $\omega_{\pi'}$ defined in \ref{ss:star-omega}, which
  takes $\Schur(\pi')$ isomorphically onto $\Schur(\pi)$ and
  interchanges $E_j$ with $F_j$ for all $j \in I$. When the
  isomorphism is applied to relation (a) above, which holds in the
  algebra $\Schur(\pi')$, we obtain the relation (b) in the algebra
  $\Schur(\pi)$.
\end{proof}

\section{Semisimplicity in higher ranks}\label{sec:general-case}
\noindent
We consider a general root datum of rank at least 2.  We put $\Schur =
\Schur(\pi)$ where $\pi$ is a finite saturated subset of $\X^+$. We
have already shown the existence and uniqueness (up to isomorphism) of
simple modules of highest weight $\lambda$, for each $\lambda \in
\pi$. First we finish the classification of the simple
$\Schur$-modules, by showing that there are no others besides the ones
constructed thus far.  This implies that $\Schur$ is semisimple.  The
argument is nearly self-contained, but we do need two well known facts
from the representation theory of complex semisimple Lie algebras: the
classification of the finite dimensional simple modules, and Weyl's
theorem on complete reducibility.

Given any fixed $i \in I$, let $\Schur^{(i)}$ be the rank 1 subalgebra
(see \ref {ss:Si}) of $\Schur$ generated by $E_i$, $F_i$ along with
all $1_n$ for $n \in \pm \pi^{(i)}$.

\begin{prop}
  Let $\Schur = \Schur(\pi)$.  If $L$ is a simple left $\Schur$-module
  then $L$ is generated by a highest weight vector of some dominant
  weight $\lambda \in \pi$. Thus $L$ is isomorphic to the simple
  quotient of $\Delta(\lambda)$.
\end{prop}

\begin{proof}
  Since $L = \bigoplus_{\lambda \in W\pi} 1_\lambda L$ it is clear that
  $L$ contains a highest weight vector, say $x_0$, of weight $\lambda
  \in W\pi$.  Since $L$ is simple, it is generated by $x_0$. For each
  $i \in I$ we restrict the action of $\Schur$ on $L$ to the rank 1
  subalgebra $\Schur^{(i)}$. Then $x_0$ is an $\Schur^{(i)}$-maximal
  vector of weight $n = \bil{\alpha_i^\vee}{\lambda}$, so by Lemma
  \ref{lem:maxvec} the $\Schur^{(i)}$-submodule it generates has a
  unique simple quotient. By Proposition \ref{prop:r1ss},
  $\Schur^{(i)}$ is a semisimple algebra, so in fact that submodule is
  already simple as a $\Schur^{(i)}$-module, and thus $n =
  \bil{\alpha_i^\vee}{\lambda} \ge 0$. Since this holds for each $i
  \in I$, we have shown that $\lambda$ is a dominant weight. In other
  words, $\lambda \in W\pi \cap \X^+ = \pi$, as desired.
\end{proof}

This result completes the classification of the simple
$\Schur$-modules. For each $\lambda \in \pi$, there is a unique (up to
isomorphism) simple $\Schur$-module, which appears as the unique
simple quotient of $\Delta(\lambda)$, and this gives a complete set of
isomorphism classes of simple $\Schur$-modules.

\subsection{}\label{ss:Bourbaki}
According to \cite[Chapter 2, \S7.10]{Bourbaki} (see Corollary 2 after
Proposition 26): \emph{If $A$ is any integral domain and $Q$ is its
  field of fractions, then for any $A$-submodule $M$ of a vector space
  $V$ over $Q$ the natural map $Q \otimes_A M \to V$ (given by $\sum
  a_j \otimes m_j \mapsto \sum a_jm_j$) is injective}. 

In the above situation, we note that $M$ is torsion-free over $A$.
Thus, if $A$ is a principal ideal domain and $M$ is finitely generated
over $A$ then it follows that $M$ is free as an $A$-module.

\begin{lem}\label{lem:int}
  Let $A$ be an integral domain and $Q$ its field of
  fractions. Suppose that $M$ is an $A$-submodule of $V$, where $V$ is
  a vector space over $Q$.

  {\rm(a)} The natural map $Q \otimes_A M \to V$ is an
    isomorphism of vector spaces if and only if $M$ contains a subset
    which spans $V$ over $Q$.

  {\rm\rm(b)} Assume that {\rm(a)} holds. If in addition $M$ is free
    over $A$ then any $A$-basis of $M$ must be also a $Q$-basis of $V$
    (so $M$ is a lattice in $V$).
\end{lem}

\begin{proof}
  (a) The natural map $Q \otimes_A M \to V$ is surjective if and only
  if $M$ contains a subset that spans $V$ over $Q$.

  (b) Let $\mathcal{B}$ be an $A$-basis of $M$. By assumption $M$
  contains a subset $\mathcal{C}$ which spans $V$ over $Q$. But each
  element of $\mathcal{C}$ is expressible as some $A$-linear
  combination of elements of $\mathcal{B}$, so $V$ is spanned by
  $\mathcal{B}$ over $Q$.

  It remains to show that $\mathcal{B}$ is linearly independent over
  $Q$. Suppose that 
  \[
    r_1 b_1 + r_2 b_2 + \cdots + r_n b_n = 0
  \]
  in $V$, for $b_j \in \mathcal{B}$ and $r_j \in Q$. Write $r_j =
  {a_j}/{b_j}$ for $a_j, b_j \in A$ with $b_j \ne 0$. Multiplying by
  $P = b_1 \cdots b_n$ we obtain an equation
  \[
    P\, r_1 b_1 + P\, r_2 b_2 + \cdots + P\, r_n b_n = 0
  \]
  in which the coefficients $P\, r_j \in A$. The linear independence
  of $\mathcal{B}$ over $A$ implies that $P\,r_j = 0$ for all $j$.
  Since $P \ne 0$ this forces $r_j = 0$ for all $j$, as desired.
\end{proof}

\begin{thm}\label{thm:simplicity}
  Let $\Schur = \Schur(\pi)$.  If $\lambda_0$ is a maximal element in
  $\pi$ (with respect to the partial order $\ldomeq$) then the left
  ideal $\Schur 1_{\lambda_0}$ (and also the right ideal
  $1_{\lambda_0}\Schur$) is a simple $\Schur$-module of highest weight
  $\lambda_0$. The formal character of $\Delta(\lambda_0)$ is given by
  Weyl's character formula.
\end{thm}

\begin{proof}
(Similar to Sections 5.12--5.15 of \cite{Jantzen:LQG}.)  
Put $\AQ= \Q[v,v^{-1}]$. Write $M(\lambda_0) = \Schur 1_{\lambda_)}$,
which is a highest weight module of highest weight $\lambda_0$. Let
$L(\lambda_0)$ be its simple quotient.  Clearly $L(\lambda_0)$ is
generated by a highest weight vector of weight $\lambda_0$.
Throughout the following argument, we let $V$ be either $M(\lambda_0)$
or $L(\lambda_0)$. Then $V = \Schur^- x_0$ where $x_0$ is a maximal
vector in $V$, so $V$ is the $\Q(v)$-linear span of elements of the
form $F_{B} x_0$ for various sequences ${B} = (i_1, \dots, i_r)$ in
$I^*$. Write $\wt({B}) = \sum_j \alpha_{i_j}$.  Let ${}_\AQ V$
be the ${}_\AQ \Schur$-submodule of $V$ generated by the maximal
vector $x_0$.  Then
\[
  {}_\AQ V = \sum_{{B}} \AQ F_{{B}} x_0 \quad\text{and}
  \quad {}_\AQ V_{\mu} = \sum_{\wt({B}) = \lambda_0 - \mu} \AQ
  F_{{B}} x_0
\]
for any $\mu \in W\pi$.  As $\AQ$-modules, both ${}_\AQ V$ and ${}_\AQ
V_{\mu}$ are finitely generated and torsion-free. Hence both ${}_\AQ V$
and ${}_\AQ V_{\mu}$ are free of finite rank over $\AQ$. Clearly ${}_\AQ
V = \sum_\mu {}_\AQ V_{\mu}$ so we get that ${}_\AQ V = \bigoplus_{\mu
  \in W\pi} {}_\AQ V_{\mu}$.  The natural map $\Q(v) \otimes_\AQ ({}_\AQ
V_{\mu}) \to V_\mu$ is surjective, for any $\mu \in W\pi$, since
$V_\mu$ is spanned by all $F_{{B}} x_0$ with $\wt({B}) =
\lambda_0 - \mu$. (Here we are writing $V_\mu$ for the weight space
$1_\mu V$ in $V$.) By Lemma \ref{lem:int} it is injective as well,
hence an isomorphism.  It follows that a basis for ${}_\AQ V_{\mu}$
over $\AQ$ is also a basis of $V_\mu$ over $\Q(v)$, and thus
\[
   \rk({}_\AQ V_\mu) = \dim_{\Q(v)} V_\mu \quad (\text{any } \mu \in
   W\pi).
\] 

We claim that the $\AQ$-module ${}_\AQ V$ is stable under the action of
the $E_j$, $F_j$, and $1_\mu$ for any $j \in I$ and any $\mu \in
W\pi$. This is obvious in the case of the $F_j$ and $1_\mu$, since
\[
F_j (F_{B} x_0) \in V_A \quad\text{and} \quad 1_\mu
(F_{B} x_0) = F_{B} 1_{\mu+\wt({B})} x_0
\]
is zero if $\mu+\wt({B}) \ne \lambda_0$ and is $F_{B} x_0$
if $\mu+\wt({B}) = \lambda_0$. Moreover, for the $E_j$ we
have by the defining relations \ref{ss:defrels}(b), (c) that
\[
\begin{aligned}
  E_j F_{B} &= E_j (F_{i_r} \cdots F_{i_1} x_0)\\ &= \sum_{1\le a \le
    r;\, i_a = j} F_{i_r} \cdots F_{i_{a+1}} \sum_{\mu \in W\pi}
  [\bil{\alpha_j^\vee}{\mu}]_j 1_\mu F_{i_{a-1}} \cdots F_{i_1} x_0
\end{aligned}
\]
and the claim follows by the preceding remarks and the fact that
$[\bil{\alpha_j^\vee}{\mu}]_j \in \AQ$ for any $\mu\in W\pi$, $j \in
I$. 

Now there is a unique homomorphism $\varphi$ of $\Q$-algebras mapping
$\AQ = \Q[v,v^{-1}]$ to $\C$ that sends $v$ and $v^{-1}$ to $1$.
Regard $\C$ as an $\AQ$-module via $\varphi$, and put
\[
  \overline{V} = \C \otimes_\AQ ({}_\AQ V) \quad \text{and}\quad
  \overline{V}_{\mu} = \C \otimes_\AQ ({}_\AQ V_{\mu})
\]
for any $\mu \in W\pi$. Then we have the direct sum decomposition
$\overline{V} = \bigoplus_\mu \overline{V}_\mu$, where each
$\overline{V}_\mu$ is a complex vector space with
\[
   \dim_\C \overline{V}_\mu = \rk ({}_\AQ V_{\mu}) = \dim_{\Q(v)}
   V_\mu.
\]
The actions of $E_i$, $F_i$, and $1_\mu$ on $V_A$ yield linear
endomorphisms of $\overline{V}$ that we denote by $e_i$, $f_i$, and
$\iota_\mu$. We put
\[
  \overline{h}_i = \textstyle \sum_{\mu \in W\pi}
  \bil{\alpha_i^\vee}{\mu} \iota_\mu
\]
for any $i \in I$. 

We claim that the endomorphisms $e_i, f_i, \overline{h}_i$ satisfy
Serre's presentation for the finite dimensional semisimple Lie algebra
$\g$ (see \ref{ss:Lie-algebra}) defined by the Cartan datum. Since the
idempotent linear operators $\iota_{\lambda_0}$ commute and are pairwise
orthogonal, it follows that $\overline{h}_i$ commutes with
$\overline{h}_j$; thus
\[
   [\overline{h}_i, \overline{h}_j] = 0, \quad \text{any } i, i \in I.
\]
We have $\varphi( [a]_i ) = [a]_{v=1} = a$ for any integer $a$ and any
$i \in I$, so from defining relations \ref{ss:defrels}(b) for the
algebra $\Schur = \Schur(\pi)$ we have
\[
   [e_i, f_j] = \delta_{i,j} \textstyle \sum_{\mu \in W\pi}
   \bil{\alpha_i^\vee}{\mu} \iota_\mu = \overline{h}_i .
\]
Recalling the convention that $1_\mu = 0$ for any $\mu \notin W\pi$ we
put also $\iota_\mu = 0$ for any $\mu \notin W\pi$. Then we can write
$\overline{h}_i = \sum_{\mu\in \X} \bil{\alpha_i^\vee}{\mu} \iota_\mu$
(which is still a finite sum) and by defining relation
\ref{ss:defrels}(c) we have
\begin{align*}
  [\overline{h}_i, e_j] &= \overline{h}_i e_j - e_j \overline{h}_i
  = \textstyle\sum_{\mu \in \X} \bil{\alpha_i^\vee}{\mu} \iota_\mu e_j -
  \sum_{\mu \in \X} \bil{\alpha_i^\vee}{\mu} e_j \iota_\mu \\
  &= \textstyle\sum_{\mu \in \X} \bil{\alpha_i^\vee}{\mu} \iota_\mu e_j -
  \sum_{\mu \in \X} \bil{\alpha_i^\vee}{\mu}  \iota_{\mu+\alpha_j} e_j
\end{align*}
and by replacing $\mu$ by $\mu-\alpha_j$ in the second sum we obtain
\begin{align*}
  [\overline{h}_i, e_j] &= \textstyle\sum_{\mu \in \X}
  \bil{\alpha_i^\vee}{\mu} \iota_\mu e_j - \sum_{\mu \in \X}
  \bil{\alpha_i^\vee}{\mu-\alpha_j} \iota_{\mu} e_j\\
  &= \textstyle\sum_{\mu \in \X} \bil{\alpha_i^\vee}{\alpha_j} \iota_\mu e_j\\
  &= \bil{\alpha_i^\vee}{\alpha_j} e_j
\end{align*}
where we have used the second part of relation \ref{ss:defrels}(a)
to get the last line. A similar calculation proves that
\[
  [\overline{h}_i, f_j] = -\bil{\alpha_i^\vee}{\alpha_j} f_j.
\]
Finally, we have 
\[ \textstyle
\varphi\big(\sqbinom{a}{n}_i\big) = \sqbinom{a}{n}_{v=1} = \binom{a}{n}
\] 
for any integers $a, n$ with $n \ge 0$. Thus the quantum Serre
relations in Theorem \ref{thm:q-Serre} imply that
\begin{align*}
\sum_{s=0}^{1-a_{ij}} (-1)^s \tbinom{1-a_{ij}}{s}
e_i^{1-a_{ij}-s}e_j e_i^s = 0 \quad (i\ne j);\\
\sum_{s=0}^{1-a_{ij}} (-1)^s \tbinom{1-a_{ij}}{s}
f_i^{1-a_{ij}-s}f_j f_i^s = 0 \quad (i\ne j)
\end{align*}
where $a_{ij} = \bil{\alpha_i^\vee}{\alpha_j}$. Thus the claim is
proved.

Now let $x_i, y_i, h_i$ ($i \in I$) be a Chevalley system of
generators for the Lie algebra $\g$.  Then by the claim of
the preceding paragraph it follows that the map $\g \to
\mathfrak{gl}(\overline{V})$ given by $x_i \to e_i$, $y_i \to f_i$,
$h_i \to \overline{h}_i$ is a homomorphism of Lie algebras, so
$\overline{V}$ is a $\mathfrak{g}$-module.

All of the preceding discussion applies equally well to $V =
L(\lambda_0)$ or $V = M(\lambda_0) = \Schur 1_{\lambda_0}$. In either
case we now see easily that $\overline{V}$ is a simple
$\g$-module of highest weight $\lambda_0$. This follows from
Weyl's theorem on complete reducibility of finite dimensional
representations of semisimple Lie algebras, which implies that
$\overline{V}$ is completely reducible, and the observation that
$\overline{V}$ is (in both cases under consideration) generated by
a maximal vector, and hence has a unique simple quotient.

It follows that the weight space dimensions in $\overline{V}$ are
given by Weyl's character formula, in both cases $V = L(\lambda_0)$ and
$V = M(\lambda_0)$. In particular, this shows that
\[
  \dim_{\Q(v)} L(\lambda_0) = \dim_{\Q(v)} M(\lambda_0). 
\]
Since $L(\lambda_0)$ is a homomorphic image of $M(\lambda_0)$, it
follows that $L(\lambda_0) = M(\lambda_0)$ and we have obtained the
result.
\end{proof}

\begin{cor}
Let $\Schur = \Schur(\pi)$. 

(a) For any $\lambda \in \pi$, the modules $\Delta(\lambda)$ and
$\Delta(\lambda)^*$ are simple as left and right $\Schur$-modules.
The formal character of $\Delta(\lambda)$ is given by Weyl's character
formula, for each $\lambda \in \pi$.

(b) If $\lambda_0$ is any maximal element of $\pi$ then $\Schur
1_{\lambda_0} \cong \Delta(\lambda_0)$ and $1_{\lambda_0}\Schur \cong
\Delta(\lambda_0)^*$.
\end{cor}

\begin{proof}
(a) By the theorem with $\pi$ replaced by $\pi_\lambda = \{\mu\in \X^+
  \colon \mu \ldomeq \lambda \}$, we know that $\Delta(\lambda)$ and
  $\Delta(\lambda)^*$ are simple as left and right
  $\Schur(\pi_\lambda)$-modules.  They may be regarded as
  $\Schur(\pi)$-modules via the natural quotient map $\Schur(\pi) \to
  \Schur(\pi_\lambda)$ defined in \ref{ss:quotient-morphisms}.  Since
  any $\Schur(\pi)$-submodule must be also an
  $\Schur(\pi_\lambda)$-submodule, it follows that they are simple
  when regarded as $\Schur(\pi)$-modules.

(b) The first isomorphism is clear because $\Delta(\lambda)$ and
  $\Schur 1_{\lambda_0}$ are both simple highest weight modules of
  highest weight $\lambda_0$. The second follows from the first, or
  repeat the argument.
\end{proof}

Now that we know the $\Delta(\lambda)$ are simple modules for all
$\lambda \in \pi$, we are ready to prove semisimplicity of $\Schur =
\Schur(\pi)$.

\begin{prop}\label{prop:ss}
  The algebra $\Schur = \Schur(\pi)$ is semisimple, and a complete set
  of isomorphism classes of simple $\Schur$-modules is given by $\{
  \Delta(\lambda) \colon \lambda \in \pi \}$. 
\end{prop}

\begin{proof} 
Similar to the proof of Proposition \ref{prop:r1ss}.  We proceed by
induction on the cardinality of $\pi$. If $\pi= \{ \lambda_0 \}$ is a
singleton then $\Schur = \Schur 1_{\lambda_0} \Schur$ and by Corollary
\ref{cor:kernel}(b) it follows that
\[
  \dim \Schur \le (\dim \Delta(\lambda_0))^2.
\]
On the other hand,the standard theory of
finite dimensional algebras implies that
\[
  \dim \Schur \ge (\dim \Delta(\lambda_0))^2.
\]
This proves equality of dimensions, and the claims follow in case
$|\pi|=1$.

Assume now that $|\pi| \ge 2$. Let $\lambda_0$ be a maximal element of
$\pi$, and put $\pi' = \pi - \{\lambda_0\}$. By induction the algebra
$\Schur(\pi')$ is semisimple and $\dim_{\Q(v)} \Schur(\pi') =
\sum_{\lambda \in W\pi'} (\dim_{\Q(v)} \Delta(\lambda))^2$. By
Corollary \ref{cor:kernel}(a) the kernel of the natural quotient map
$p_{\pi, \pi'}$ is equal to $\Schur 1_{\lambda_0} \Schur$. By
Corollary \ref{cor:kernel}(b) again we have $\dim_{\Q(v)} \ker
p_{\pi,\pi'} \le (\dim_{\Q(v)} \Delta(\lambda_0))^2$, so
\begin{align*}
  \dim_{\Q(v)} \Schur &= \dim_{\Q(v)} \ker p_{\pi,\pi'} +
  \textstyle \sum_{\lambda \in W\pi'} (\dim_{\Q(v)} \Delta(\lambda))^2\\
  & \le \textstyle \sum_{\lambda \in W\pi} (\dim_{\Q(v)} \Delta(\lambda))^2.
\end{align*}
On the other hand, since the $\Delta(\lambda)$ are pairwise
non-isomorphic simple modules (indeed, no two of them has the same
highest weight), the standard theory of finite dimensional algebras
implies that
\[
  \textstyle \dim_{\Q(v)} \Schur \ge \sum_{\lambda \in \pi} (\dim_{\Q(v)}
  \Delta(\lambda))^2.
\]
Hence the dimensions agree, and the claims follow.
\end{proof}

\begin{cor}\label{cor:mult-iso}
  Let $\Schur=\Schur(\pi)$ and let $\lambda_0$ be maximal in
  $\pi$. The natural multiplication map $\Delta(\lambda_0) \otimes
  \Delta(\lambda_0)^* \to \Schur 1_{\lambda_0} \Schur$ considered in
  Corollary \ref{cor:kernel}(b) is actually an isomorphism of
  $\Schur$-bimodules.
\end{cor}

\begin{proof}
In the situation considered in the proposition, we have $\ker
p_{\pi,\pi'} = \Schur 1_{\lambda_0} \Schur$. It follows from the
preceding result that the kernel of $p_{\pi,\pi'}$ must have dimension
$(\dim_{\Q(v)} \Delta(\lambda_0))^2$. Since the multiplication map
$
  \Delta(\lambda_0) \otimes_{\Q(v)} \Delta(\lambda_0)^* \to \Schur
  1_{\lambda_0} \Schur
$
in Corollary \ref{cor:kernel}(b) is surjective, the result follows by a
dimension comparison.
\end{proof}

\section{Constructing cellular bases of $\Schur(\pi)$}%
\label{sec:cellular}\noindent
We now show how fixing an arbitrary basis of the simple module
$\Delta(\lambda)$, for each $\lambda \in \pi$, leads to a cellular
basis of $\Schur(\pi)$.

\subsection{Cosaturated sets}\label{ss:cosat}
Let $\pi$ be a saturated subset of $\X^+$. A subset $\Phi$ of $\pi$ is
\emph{cosaturated} if $\Phi$ can be written in the form $\pi - \pi'$,
where $\pi'$ is a saturated subset of $\pi$. Equivalently, $\Phi$ is
cosaturated if and only if it is successor-closed with respect to
$\domeq$; i.e., if $\lambda \in \Phi$ and $\mu \domeq \lambda$ where
$\mu \in \pi$ then $\mu \in \Phi$.

Examples of cosaturated subsets of $\pi$ are the subsets $\pi[\dom
  \lambda] = \{ \mu \in \pi: \mu \dom \lambda\}$ and $\pi[\domeq
  \lambda] = \{ \mu \in \pi: \mu \domeq \lambda\}$.

Given any cosaturated subset $\Phi$ of $\pi$ let $\Schur[\Phi] =
\sum_{\mu \in \Phi} \Schur 1_\mu \Schur$ be the ideal generated by all
idempotents $1_\mu$ such that $\mu \in \Phi$.  In particular, if
$\Phi = \pi[\dom \lambda]$ or $\pi[\domeq \lambda]$ then we write
$\Schur[\dom \lambda]$ or $\Schur[\domeq \lambda]$ (respectively) for
the ideal $\Schur[\Phi]$.

\begin{lem}\label{lem:rat-filt} 
  Let $\Phi$ be any cosaturated subset of $\pi$. Write $\Phi = \pi -
  \pi'$ where $\pi'$ is a saturated subset of $\pi$. Then
  $\Schur[\Phi]$ is equal to the kernel of the natural homomorphism
  $p_{\pi, \pi'}$.
\end{lem}

\begin{proof}
We proceed by induction on the cardinality of $\Phi$. If $\Phi =
\{\lambda_0\}$ then $\lambda_0$ is necessarily maximal in $\pi$, and
the result is clear from Corollary \ref{cor:kernel}(a). Assume that
$|\Phi| \ge 2$.  Then we can write $\Phi = \Phi' \cup \{\lambda\}$
where $\lambda$ is minimal in $\Phi$ (there does not exist any $\mu
\in \Phi$ such that $\mu \dom \lambda$) and $\Phi' = \Phi -
\{\lambda\}$. Then $\Phi'$ is also cosaturated in $\pi$ and its
complement $\pi' \cup \{\lambda\}$ is saturated. The diagram of
projections
\[
\begin{gathered}
\xymatrix{ \Schur(\pi) \ar[dr]_{p_{\pi,\pi'\cup \{\lambda\}}}
  \ar[rr]^{p_{\pi,\pi'}} && \Schur(\pi') \\ &
    \Schur(\pi' \cup \{\lambda\}) \ar[ur]_{p_{\pi'\cup \{\lambda\},\pi'}}
}
\end{gathered}
\]
commutes. By induction, the kernel of $p_{\pi,\pi'\cup \{\lambda\}}$
is the ideal $\Schur[ \Phi' ]$ generated by all $1_\mu$ with $\mu \in
\Phi'$. By Corollary \ref{cor:kernel}(a) again the kernel of $p_{\pi'\cup
  \{\lambda\},\pi'}$ is the ideal of $\Schur(\pi' \cup \{\lambda\})$
generated by $1_\lambda$. It follows that the kernel of $p_{\pi,\pi'}$
is generated by all $1_\mu$ as $\mu$ varies over the set $\Phi = \Phi'
\cup \{\lambda\}$. This completes the proof.
\end{proof}

\begin{prop}\label{prop:rat-basis}
  For each $\lambda \in \pi$, let $\pi_\lambda = \{\mu \in \pi: \mu
  \ldomeq \lambda\}$ and put $d(\lambda) = \dim_{\Q(v)}
  \Delta(\lambda)$. Pick elements $\overline{x}_1, \dots,
  \overline{x}_{d(\lambda)} \in \Schur(\pi_\lambda)^-$ such that the
  set 
  \[
    \{ \overline{x}_s 1_\lambda: 1 \le s \le d(\lambda)\}
  \] 
  is a $\Q(v)$-basis of $\Delta(\lambda) = \Schur(\pi_\lambda)^-
  1_\lambda$. For each $s = 1, \dots, d(\lambda)$ let $x_s \in
  \Schur(\pi)^-$ be any preimage of $\overline{x}_s$ under the
  quotient map $p_{\pi, \pi_\lambda}$. Put $C^\lambda_{s,t} = x_s
  1_\lambda x^*_t$, for any $1 \le s,t \le d(\lambda)$.  Then the set
  \[
    B(\pi) = \textstyle \bigsqcup_{\lambda \in \pi} \{ C^\lambda_{s,t} :
    1 \le s,t \le d(\lambda)\}
  \]
  is a basis of $\Schur = \Schur(\pi)$ over $\Q(v)$.
\end{prop}

\begin{proof}
If $\pi = \{\lambda_0\}$ has cardinality one we have $\Schur = \Schur
1_{\lambda_0} \Schur$ and the claim follows immediately from the
isomorphism of Corollary \ref{cor:mult-iso}. 

Proceeding by induction on the cardinality of $\pi$, we now assume
that $|\pi| \ge 2$. Pick any maximal element $\lambda_0$ of $\pi$ and
put $\pi' = \pi - \{\lambda_0\}$.  For any $\lambda \in \pi'$, the
quotient map $p_{\pi, \pi_\lambda}$ factors through $\Schur(\pi')$:
i.e., the diagram
\[
\begin{gathered}
\xymatrix{
  \Schur(\pi)\ar[dr]_{p_{\pi,\pi'}} \ar[rr]^{p_{\pi,\pi_\lambda}} && \Schur(\pi_\lambda) \\
           &  \Schur(\pi') \ar[ur]_{p_{\pi',\pi_\lambda}}
} 
\end{gathered}\tag{b}
\]
commutes.  Thus by the inductive hypothesis the set
\[
   B(\pi') = \textstyle \bigsqcup_{\lambda \in \pi'} \{
   p_{\pi,\pi'}(C^\lambda_{s,t}) : 1 \le s, t \le d(\lambda) \}
\]
is a basis of the algebra $\Schur(\pi')$. We will now show that the
set $B(\pi)$ is linearly independent. Suppose that there are scalars
$\omega^\lambda_{s,t} \in \Q(v)$ such that some linear combination of
the elements in $B(\pi)$ satisfies
\[
  \textstyle \sum_{\lambda \in \pi} \sum_{1 \le s,t \le d(\lambda)}
  \omega^\lambda_{s,t} \,C^\lambda_{s,t} = 0. \tag{c}
\]
By applying the linear map $p_{\pi,\pi'}$ to the equality (c) above we
get
\[
  \textstyle \sum_{\lambda \in \pi'} \sum_{1 \le s,t \le d(\lambda)}
  \omega^\lambda_{s,t} \, p_{\pi,\pi'}(C^\lambda_{s,t}) = 0.
\]
By the induction hypothesis $B(\pi')$ is a basis, so all the
$\omega^\lambda_{s,t} = 0$, for $\lambda \in \pi'$, $1 \le s,t \le
d(\lambda)$. Hence the original linear combination in equation (c) 
reduces to a linear combination of the form
\[
  \textstyle \sum_{1 \le s,t \le d(\lambda_0)}
  \omega^{\lambda_0}_{s,t} \, C^{\lambda_0}_{s,t} = 0.
\]
This forces all the $\omega^{\lambda_0}_{s,t} = 0$, for $1 \le s,t \le
d(\lambda_0)$, since by Corollary \ref{cor:mult-iso} the set $\{
C^{\lambda_0}_{s,t}: 1 \le s,t \le d(\lambda_0) \}$ is a $\Q(v)$-basis
of the ideal $\Schur 1_{\lambda_0} \Schur$. This proves the desired
linear independence of the set $B(\pi)$.

Since by Proposition \ref{prop:ss} the cardinality of $B(\pi)$ is
equal to the dimension of $\Schur=\Schur(\pi)$, the set $B(\pi)$ is a
basis of $\Schur$, and the proof is complete.
\end{proof}

Now we can show that the basis constructed in the preceding
proposition is cellular. This is the main result we have been aiming
towards.

\begin{thm}\label{thm:cellular}
  The basis $B(\pi)$ in Proposition \ref{prop:rat-basis} is a cellular
  basis of $\Schur = \Schur(\pi)$, with respect to the anti-involution
  $*$ defined in \ref{ss:star-omega}.
\end{thm}

\begin{proof}
We have to show that the basis $B(\pi)$ satisfies the conditions of
Definition 1.1 in \cite{GL}. One of the requirements is that
$(C^\lambda_{s,t})^* = C^\lambda_{t,s}$ for all $\lambda, s, t$, which
is clear by the definitions.

Fix $\lambda \in \pi$. Consider the cosaturated subset $\Phi =
\pi[\dom \lambda]$, and put $\pi' = \pi - \Phi$. Then $\pi'$ is a
saturated subset of $\pi$ and $\lambda$ is a maximal element of
$\pi'$. We have by Corollary \ref{cor:mult-iso} an isomorphism
\[
  \Delta(\lambda) \otimes_{\Q(v)} \Delta^*(\lambda) \to \Schur(\pi')
  1_\lambda \Schur(\pi') \tag{a}
\]
induced by multiplication. Since $\{\overline{x}_s 1_\lambda : 1 \le s
\le d(\lambda)\}$ is a basis of the left ideal $\Delta(\lambda) =
\Schur(\pi_\lambda)1_\lambda$, it follows that for any $a \in
\Schur(\pi)$ we have unique scalars $r_{a}(s,s') \in \Q(v)$ such that
\[
  a \cdot \overline{x}_s 1_\lambda = \textstyle
  \sum_{s'=1}^{d(\lambda)} r_a(s,s') \, \overline{x}_{s'}1_\lambda. \tag{b}
\]
It follows from the isomorphism in (a) that in the algebra
$\Schur(\pi')$ we have equalities
\[
  p_{\pi,\pi'}(a) \cdot p_{\pi,\pi'}(C^\lambda_{s,t}) =\textstyle
  \sum_{s'=1}^{d(\lambda)}  r_a(s,s')\, p_{\pi,\pi'}(C^\lambda_{s',t}). \tag{c}
\]
Evidently, the coefficient $r_a(s,s')$ is independent of $t$. Now we
have an algebra isomorphism $\Schur(\pi) / \Schur[\dom \lambda] \simeq
\Schur(\pi')$ induced by the map $p_{\pi,\pi'}$. Reading the 
equality (c) via this isomorphism gives the congruence
\[
  a \cdot C^\lambda_{s,t} \equiv \textstyle \sum_{s'=1}^{d(\lambda)}
  r_a(s,s')\, C^\lambda_{s',t} \pmod{\Schur[\dom \lambda]}
\]
where, as before, $r_a(s,s')$ is independent of $t$. This completes
the proof.
\end{proof}

The following is an immediate restatement of the preceding theorem. 

\begin{cor}\label{prop:rat-filtration}
  Let $\pi = \{\lambda_1, \lambda_2, \dots, \lambda_m\}$ be any
  enumeration of $\pi$ such that each $\lambda_j$ is maximal in
  $\{\lambda_j, \dots, \lambda_m\}$, for all $j = 1, \dots, m$.  Then
  each $\Phi_j = \{\lambda_1, \dots, \lambda_j\}$ is cosaturated in
  $\pi$ and 
  \[
    0 \subset \Schur[\Phi_1] \subset \cdots \subset \Schur[\Phi_{m-1}]
    \subset \Schur[\Phi_m] = \Schur
  \]
  is an increasing filtration of $\Schur = \Schur(\pi)$ by two-sided
  ideals. Each $\Schur[\Phi_j]$ is invariant under the anti-involution
  $*$, and 
  \[
    \textstyle \bigsqcup_{\lambda \in \Phi_j} \{C^\lambda_{s,t} : 1 \le
    s,t \le d(\lambda)\}
  \]
  is a basis of $\Schur[\Phi_j]$ over $\Q(v)$, for each $j = 1, \dots,
  m$.
\end{cor}

The increasing filtration in the corollary is a defining cell chain
for the cellular structure on $\Schur(\pi)$, in the sense of K\"{o}nig
and Xi \cite{KX1, KX3}.

\begin{example}\label{ex:r1cellular-basis} 
Suppose the root datum has rank one, and $\pi$ is a finite saturated
subset of $\X^+ = \Z_{\ge 0}$. Then $\{F_i^{(b)}1_n : 0 \le b \le n\}$
is a basis of $\Delta(n)$, for each $n \in \pi$, so the set
\begin{equation*}
  B(\pi) = \textstyle \bigsqcup_{n \in \pi} \{ F_i^{(b)} 1_n
  E_i^{(a)} \colon 0 \le a,b \le n \}
\end{equation*}
is the cellular basis of Theorem \ref{thm:cellular} for $\Schur(\pi)$.
\end{example}

\subsection{}\label{ss:r1canonical-basis}
It is interesting to compare the cellular basis of Example
\ref{ex:r1cellular-basis} to Lusztig's canonical basis of the modified
form $\dot{\mathbf{U}}(\mathfrak{sl}_2)$ of the quantized enveloping
algebra $\mathbf{U}(\mathfrak{sl}_2)$. Let $n \in \pi$. If $a+b \ge n$
then
\begin{equation}
  F_i^{(b)} 1_n E_i^{(a)} = \textstyle \sum_{t\ge 0} \sqbinom{a+b-n}{t}_i \,
  E_i^{(a-t)} 1_{n-2(a+b-t)}F_i^{(b-t)}. 
\end{equation}
This follows from Lemma \ref{lem:commutation} by an easy calculation.
This shows that for all $a+b \ge n$ the element $F_i^{(b)} 1_n
E_i^{(a)}$ is equal to the element $E_i^{(n-b)} 1_{-n} F_i^{(n-a)} =
E_i^{(a')} 1_{-n} F_i^{(b')}$ modulo terms in $\Schur[\dom n]$, where
$a' = n-b$, $b'=n-a$. Note that the condition $a+b \ge n$ is
equivalent to the condition $a'+b' \le n$. Thus we obtain a different
cellular basis of $\Schur$ of the form
\begin{equation}
  \textstyle \bigsqcup_{n \in \pi} \big(\{ F_i^{(b)} 1_n E_i^{(a)} \colon a+b
  \le n \} \cup \{ E_i^{(a')} 1_{-n} F_i^{(b')} \colon a'+b' \le n \}\big)
\end{equation}
which has a unitriangular relation with the original basis. Note that
the union in (b) is not disjoint: when $a+b = n$ we have from (a) the
equality $F_i^{(b)} 1_n E_i^{(a)} = E_i^{(a)} 1_{-n} F_i^{(b)}$. The
basis in (b) is the same as Lusztig's canonical basis in rank one;
compare with Prop.~25.3.2 in \cite{Lusztig:book}.

\section{Specialization: cellular bases of $\Schur_q(\pi)$}%
\label{sec:cellular-A}\noindent
Now we consider specializations $\Schur_q(\pi)$ with respect to the
$\Schur(\pi)$-analogue of the Lusztig divided power integral form of a
quantized enveloping algebra. We will show that an analogue of
Theorem \ref{thm:cellular} holds for $\Schur_q(\pi)$. The argument for
this is less self-contained, and from now on we need to assume some
facts that (seem to) depend on properties of the canonical basis.

\subsection{}
We summarize some results of Lusztig.  Fix a root datum (of finite
type) and let $\UU$ be the quantized enveloping algebra over $\Q(v)$
determined by the given root datum. It is defined by generators $E_i$,
$F_i$ ($i \in I$), $K_h$ ($h \in \X^\vee$) and relations which we
omit.  Let $\dot{\UU}$ be the modified form of $\UU$ constructed in
\cite{Lusztig:PNAS} or \cite{Lusztig:book}. This is generated by an
infinite family $1_\lambda$ ($\lambda \in \X$) of pairwise orthogonal
idempotents, along with $E_i 1_\lambda$, $F_i 1_\lambda$ ($i \in I$,
$\lambda \in \X$).  In \cite{Lusztig:book} it is shown that the
canonical basis for the plus part $\UU^+$ of $\UU$ extends to all of
$\dot{\UU}$. Theorem 29.3.3 in \cite{Lusztig:book}, which Lusztig
calls a ``refined Peter--Weyl theorem'', is essentially the statement
that this basis is an ``integral'' cellular basis of $\dot{\UU}$ in
the sense of \cite{GL}. Specifically, the canonical basis of
$\dot{\UU}$ can be written as a disjoint union $\dot{\BB} =
\bigsqcup_{\lambda \in \X^+} \dot{\BB}[\lambda]$, where
\begin{equation}
  \dot{\BB}[\lambda] = \{\textbf{b}^\lambda_{S,T}: S,T \in
  \mathcal{T}_\lambda\}.
\end{equation}
The set $\mathcal{T}_\lambda$ is any indexing set for a basis of
weight vectors of the finite dimensional simple $\dot{\UU}$-module of
highest weight $\lambda$; this could simply be the set $\{1, \dots,
d(\lambda)\}$. There is a unique $\Q(v)$-linear anti-involution $*$ on
$\dot{\UU}$ such that $1_\lambda^* = 1_\lambda$ and $(E_i 1_\lambda)^*
= 1_\lambda F_i$ for all $i \in I$, $\lambda \in \X$. For any
$\lambda, S, T$ we have
\begin{gather}
  (\mathbf{b}^\lambda_{S,T})^* = \mathbf{b}^\lambda_{T,S}\\
  u \cdot \mathbf{b}^\lambda_{S,T} = \sum_{S' \in \mathcal{T}_\lambda}
  r_u(S', S) \mathbf{b}^\lambda_{S',T} \pmod{\dot{\UU}[\dom \lambda]}
\end{gather}
for all $u \in \dot{\UU}$, where $r_u(S,S') \in \A$ is independent of
$T$. Here $\dot{\UU}[\dom \lambda]$ is the ideal spanned by
$\bigsqcup_{\mu \dom \lambda} \dot{\BB}[\mu]$; the fact that this is an
ideal is implicit in the refined Peter--Weyl theorem. 

There is an integral form ${}_\A\dot{\UU}$ of $\dot{\UU}$, defined as
the $\A$-subalgebra of $\dot{\UU}$ generated by all 
\[
  1_\lambda\ (\lambda \in \X),\quad E_i^{(m)}1_\lambda,\, F_i^{(m)}
  1_\lambda\ (i \in I, m \ge 0, \lambda \in \X).
\]
Lusztig proves that $\dot{\BB}$ is an $\A$-basis of ${}_\A\dot{\UU}$
and that the structure constants with respect to this basis all lie in
$\A$. In particular, ${}_\A \dot{\UU}$ is free as an $\A$-module and
the natural map
\[
  \Q(v) \otimes_\A ({}_\A \dot{\UU}) \to \dot{\UU}
\]
given by $a \otimes u \mapsto au$, is an isomorphism. For further
details on the above facts see Section 2 of \cite{Doty:PGSA}.

\subsection{}
For any subset $\Phi$ of $\X^+$ we define $\dot{\UU}[\Phi]$ to be the
subspace of $\dot{\UU}$ spanned over $\Q(v)$ by $\bigsqcup_{\mu \in
  \Phi} \dot{\BB}[\mu]$, and we set ${}_\A \dot{\UU}[\Phi] = {}_\A
\dot{\UU} \cap \dot{\UU}[\Phi]$. Then ${}_\A \dot{\UU}[\Phi]$ is
spanned over $\A$ by $\bigsqcup_{\mu \in \Phi} \dot{\BB}[\mu]$.

Now let $\pi$ be a given finite saturated subset of $\X^+$.  Then its
complement $\Phi = \X^+ - \pi$ is a cosaturated subset of $\X^+$ and
$\dot{\UU}[\X^+-\pi]$, ${}_\A\dot{\UU}[\X^+-\pi]$ are ideals in
$\dot{\UU}$, ${}_\A \dot{\UU}$, respectively.  From \cite{Doty:PGSA}
we have algebra isomorphisms
\begin{align}
  \Schur(\pi) &\simeq \dot{\UU}/\dot{\UU}[\X^+-\pi], \\
  {}_\A\Schur(\pi) &\simeq {}_\A\dot{\UU}/{}_\A\dot{\UU}[\X^+-\pi] .
\end{align}
The map $\dot{\UU} \to \Schur(\pi)$ giving the isomorphism (a) is the
one sending $E_i1_\lambda$ and $F_i1_\lambda$ onto the corresponding
generators of $\Schur(\pi)$, and the map ${}_\A\dot{\UU} \to
{}_\A\Schur(\pi)$ producing the isomorphism (b) is its restriction. The
kernel of each map is generated by the idempotents $1_\mu$ such that
$\mu \in \X^+ - \pi$.  

Given any subset $\Phi$ of $\X^+$ we write $\dot{\BB}[\Phi] =
\bigsqcup_{\lambda \in \Phi} \dot{\BB}[\lambda]$.  The nonzero part of
the image of $\dot{\BB}[\pi]$ is an $\A$-basis of ${}_\A\Schur(\pi)$,
and in particular, ${}_\A\Schur(\pi)$ is a free $\A$-module of rank
$\textstyle \sum_{\lambda \in \pi} d(\lambda)^2$.  Furthermore, the
natural map
\begin{equation}
  \Q(v) \otimes_\A ({}_\A \Schur(\pi)) \to \Schur(\pi)
\end{equation}
given by $a \otimes u \mapsto au$, is an isomorphism.

\subsection{}
For $\lambda \in \pi$ we may regard the simple $\Schur(\pi)$-module
$\Delta(\lambda)$ as a $\dot{\UU}$-module via the map $\dot{\UU} \to
\Schur(\pi)$. It must be simple as a $\dot{\UU}$-module. The family
$\{ \Delta(\lambda) : \lambda \in \X^+\}$ is a complete set of
representatives of the isomorphism classes of simple unital
$\dot{\UU}$-modules. (See section 23.1.4 in \cite{Lusztig:book} for
the meaning of unital in this context.) The family $\{ \Delta(\lambda)
: \lambda \in \X^+\}$ is also a complete set of representatives of the
isomorphism classes of simple integrable $\UU$-modules of type
$\mathbf{1}$. (See section 5.2 of \cite{Jantzen:LQG} for the
definition of type $\mathbf{1}$ representations.)

For any $\lambda \in \pi$ let ${}_\A\Delta(\lambda)$ be the
$\aschur$-submodule of $\Delta(\lambda)$ generated by a chosen maximal
vector $0 \ne v_\lambda \in \Delta(\lambda)$. It is known that
${}_\A\Delta(\lambda)$ is free as an $\A$-module, of rank equal to
$d(\lambda)$. Furthermore,
\begin{enumerate}
\item [(a)] ${}_\A\Delta(\lambda)$ generates $\Delta(\lambda)$ as a
  vector space over $\Q(v)$,

\item [(b)] ${}_\A\Delta(\lambda) = \bigoplus_{\mu \in W\pi} 1_\mu
  ({}_\A\Delta(\lambda))$, where $1_\mu ({}_\A\Delta(\lambda)) =
  {}_\A\Delta(\lambda) \cap 1_\mu \Delta(\lambda)$,

\item [(c)] the nonzero part of the image of $\dot{\BB}$ in
  ${}_\A\Delta(\lambda) = \aschur(\pi_\lambda)1_\lambda$ under the map
  $a \mapsto a 1_\lambda$ is an $\A$-basis of ${}_\A\Delta(\lambda)$.
\end{enumerate}
\noindent
Conditions (a) and (b) say that ${}_\A\Delta(\lambda)$ is a minimal
admissible lattice in $\Delta(\lambda)$. In particular, (a) implies
that the natural map
\[
  \Q(v) \otimes_\A ({}_\A \Delta(\lambda)) \to \Delta(\lambda),
\]
given by $a \otimes v \mapsto av$, is surjective. In fact, by Lemma
\ref{lem:int} it is an isomorphism. Furthermore, any $\A$-basis of
${}_\A \Delta(\lambda)$ is also a $\Q(v)$-basis of $\Delta(\lambda)$.

\subsection{}
Given any commutative ring $\Bbbk$ with 1 and a fixed invertible
element $q \in \Bbbk$ we regard $\Bbbk$ as an $\A$-algebra via the
ring homomorphism $\A \to \Bbbk$ mapping $v \to q$ for all $k \in
\Z$. We wish to study the algebra
\[
   \Schur_q(\pi) = {}_\Bbbk \Schur(\pi) = \Bbbk \otimes_\A
  (\aschur(\pi)).
\]
This algebra is the $\Bbbk$-form of $\Schur(\pi)$. We call it a
generalized $q$-Schur algebra. We will identify generators $E_i$,
$F_i$, and $1_\lambda$ with their images $1\otimes E_i$,
$1\otimes F_i$, and $1\otimes 1_\lambda$ in $\Schur_q(\pi)$.

In case $\Bbbk = \Q(v)$, regarded as $\A$-algebra via the natural
embedding $\A \subset \Q(v)$, we may identify $\Schur_v(\pi) =
{}_{\Q(v)} \Schur(\pi)$ with the rational form $\Schur(\pi)$.

\subsection{}
Put $\Schur = \Schur(\pi)$ and $\aschur = \aschur(\pi)$.  If $\Phi$ is
any cosaturated subset of $\pi$ we put $\aschur[\Phi] = \sum_{\mu \in
  \Phi} \aschur 1_\mu \aschur$. This is the ideal of $\aschur =
\aschur(\pi)$ generated by the idempotents $1_\mu$ such that $\mu \in
\Phi$. We have the following ``integral'' analogue of Lemma
\ref{lem:rat-filt}.

\begin{lem}\label{lem:A-filt}
  Let $\Phi$ be any cosaturated subset of $\pi$. Write $\Phi = \pi -
  \pi'$ where $\pi'$ is a saturated subset of $\pi$. Then the kernel
  of the restriction map ${}_\A p_{\pi,\pi'}$ defined in
  \ref{ss:A-quotient-morphisms} is equal to $\aschur[\Phi]$. 
\end{lem}

\begin{proof}
We will identify elements of $\dot{\BB}$ with their images in any
$\Schur(\pi)$.  We have by definition that $\dot{\BB}[\pi] =
\dot{\BB}[\pi'] \sqcup \dot{\BB}[\Phi]$.  Since $\dot{\BB}[\pi]$,
$\dot{\BB}[\pi']$ are respectively $\Q(v)$-bases of $\Schur(\pi)$,
$\Schur(\pi')$ it follows by Lemma \ref{lem:rat-filt} that the kernel
$\Schur[\Phi]$ of the map $p_{\pi,\pi'}$ is spanned over $\Q(v)$ by
$\dot{\BB}[\Phi]$. 

The kernel of the restriction map ${}_\A p_{\pi,\pi'}$ is equal to
$\aschur \cap \Schur[\Phi]$. Evidently $\aschur[\Phi] \subseteq
\aschur \cap \Schur[\Phi]$.  The reverse inclusion $\aschur[\Phi]
\supseteq \aschur \cap \Schur[\Phi]$ follows from the preceding
paragraph and the fact that $\dot{\BB}[\pi]$, $\dot{\BB}[\pi']$ are
respectively $\A$-bases of $\aschur(\pi)$, $\aschur(\pi')$.
\end{proof}

\begin{prop}\label{prop:A-basis}
  For each $\lambda \in \pi$, let $\pi_\lambda = \{\mu \in \pi: \mu
  \ldomeq \lambda\}$. Choose a maximal vector $0 \ne v_\lambda$ in
  ${}_\A\Delta(\lambda)$, and pick elements $\overline{x}_1, \dots,
  \overline{x}_{d(\lambda)} \in \aschur(\pi_\lambda)^-$ such that the
  set
  \[
    \{ \overline{x}_s v_\lambda: 1 \le s \le d(\lambda)\}
  \] 
  is an $\A$-basis of ${}_\A\Delta(\lambda) = \aschur(\pi_\lambda)^-
  1_\lambda$. For each $s$, let $x_s \in \aschur(\pi)^-$ be any
  preimage of $\overline{x}_s$ under the quotient map $p_{\pi,
    \pi_\lambda}$. Put $C^\lambda_{s,t} = x_s 1_\lambda x^*_t$, for
  any $1 \le s,t \le d(\lambda)$.  Then the set
  \[
    A(\pi) = \textstyle \bigsqcup_{\lambda \in \pi} \{ C^\lambda_{s,t} :
    1 \le s,t \le d(\lambda)\}
  \]
  is a basis of $\aschur(\pi)$ over $\A$.
\end{prop}

\begin{proof}
By Proposition \ref{prop:rat-basis} the set $A(\pi)$ is a
$\Q(v)$-basis of $\Schur = \Schur(\pi)$. Hence it is linearly
independent over $\Q(v)$, and hence linearly independent over $\A$.
So it is enough to show that $A(\pi)$ spans ${}_\A\Schur$ over $\A$. 

Pick an ordering $\lambda_1, \lambda_2, \dots, \lambda_n$ of $\pi$
such that $\lambda_j$ is a maximal element of $\pi_j = \{\lambda_j,
\dots, \lambda_n\}$ for all $j = 1, \dots, n$. Then each $\pi_j$ is a
saturated subset of $\pi$ and $\Phi_j = \pi - \pi_j$ is cosaturated.
Clearly $\aschur[\Phi_j] \subset \aschur[\Phi_{j+1}]$ for all $j<n$
and $\aschur[\Phi_n] \subset \aschur$. We claim that $\aschur[\Phi_n]
= \aschur$. To see this, observe that $\aschur[\Phi_n]$ is a left
$\aschur$-module. It contains the idempotents $1_\lambda$ for any
$\lambda \in \pi$ by definition, and it follows from an easy induction
argument that it contains $1_{w(\lambda)}$ for all $w \in W$, $\lambda
\in \pi$. Thus $1 = \sum_{\mu \in W\pi} 1_\mu \in \aschur[\Phi_n]$, so
we have the inclusion $\aschur = \aschur \cdot 1 \subseteq
\aschur[\Phi_n]$, and the claim is proved. Hence we have a filtration
\[
0 \subset \aschur[\Phi_1] \subset \aschur[\Phi_2] \subset \cdots
\subset \aschur[\Phi_n] = \aschur
\]
of $\aschur$ by two-sided ideals.  We use this filtration to prove
inductively that the set $\bigsqcup_{j=1}^k \{C^{\lambda_j}_{s,t} : 1
\le s,t \le d(\lambda_j)\}$ spans the ideal $\aschur[\Phi_k]$ for each
$k = 1, \dots, n$.

Since $\aschur[\Phi_1] = \aschur[\{\lambda_1\}] = \aschur
1_{\lambda_1}\aschur$ the claim is true in the base case. 

Assuming by induction that $\bigsqcup_{j=1}^{k-1} \{C^{\lambda_j}_{s,t}
: 1 \le s,t \le d(\lambda_j)\}$ spans the ideal $\aschur[\Phi_{k-1}]$,
the existence of the isomorphism
\[
  \aschur[\Phi_k] / \aschur[\Phi_{k-1}] \simeq {}_\A \Delta(\lambda_k)
  \otimes_\A ({}_\A \Delta(\lambda_k)^*)
\]
shows that each element of $\aschur[\Phi_k]$ is expressible as an
$\A$-linear combination of the $C^{\lambda_k}_{s,t}$ modulo terms in
the kernel $\aschur[\Phi_{k-1}]$. Hence
$
  \textstyle\bigsqcup_{j=1}^{k} \{C^{\lambda_j}_{s,t} : 1 \le s,t \le
  d(\lambda_j)\}
$
spans $\aschur[\Phi_{k}]$, as desired.  This completes the proof.
\end{proof}

\begin{cor}\label{cor:k-basis}
  The basis $A(\pi)$ of the preceding proposition is a cellular basis
  of $\aschur = \aschur(\pi)$, with respect to the anti-involution
  $*$.  For any commutative ring $\Bbbk$ which is an $\A$-algebra via
  the specialization $v \mapsto q$, for an invertible $q \in \Bbbk$,
  the image of this basis in $\Schur_q(\pi)$ is a cellular basis of
  $\Schur_q(\pi)$.
\end{cor}

\begin{proof}
By Lemma \ref{lem:int} the basis $A(\pi)$ is also a $\Q(v)$-basis of
the rational form $\Schur = \Schur(\pi)$.  By Theorem
\ref{thm:cellular} we have the equality $(C^\lambda_{s,t})^* =
C^\lambda_{t,s}$ for all $\lambda, s,t$. Furthermore, for any $a \in
\aschur$ we have
\[
  a \cdot C^\lambda_{s,t} \equiv \textstyle \sum_{s'=1}^{d(\lambda)}
  r_a(s,s')\, C^\lambda_{s',t} \pmod{\Schur[\dom \lambda]}
\]
where $r_a(s,s') \in \Q(v)$ is independent of $t$. The fact that
$A(\pi)$ is an $\A$-basis of $\aschur$ implies that the coefficients
$r_a(s,s') \in \A$ in the above expression. This proves that $A(\pi)$
is a cellular basis of $\aschur$.

For the last statement, recall that by \cite{GL} cellularity behaves
well under change of base ring.
\end{proof}

\begin{rmk}
  It is natural to ask if there is a self-contained proof of Corollary
  \ref{cor:k-basis} which does not depend on a descent from
  ${}_\A\dot{\UU}$. This appears to be an interesting problem. Indeed,
  verifying that $\aschur(\pi)$ is free over $\A$ without relying on a
  descent from ${}_\A\dot{\UU}$ seems to be difficult.
\end{rmk}

\section{A filtration on certain projective modules}%
\label{sec:filtration}\noindent
We consider a natural decomposition of the left regular representation
of $\Schur_q(\pi)$ into projective modules, coming directly from the
given family of orthogonal idempotents in the definition of $\Schur =
\Schur(\pi)$ by generators and relations.

\subsection{}\label{ss:projectives}
Throughout this section we work over an arbitrary commutative ring
$\Bbbk$ with 1, regarded as $\A$-algebra via the specialization $v
\mapsto q \in \Bbbk$, for a given invertible element $q \in \Bbbk$.
Fix a finite saturated subset $\pi$ of $\X^+$, and write $\Schur_q =
\Schur_q(\pi)$.  We have a decomposition
\begin{equation*}
  \Schur_q = \textstyle \bigoplus_{\lambda \in W\pi} \Schur_q 1_\lambda.
\end{equation*}
This is a decomposition of the left regular representation into
projective left modules.  We aim to construct certain natural
filtrations on these projectives. We will work with a fixed maximal
element $\lambda_1$ in $\pi$, and proceed inductively.

We write $\Delta_q(\lambda)$ for the $\Schur_q$-module $\Bbbk
\otimes_\A ({}_\A \Delta(\lambda)$ for each $\lambda$. If $\Bbbk$ is a
field, $\Delta_q(\lambda)$ is a $q$-analogue of a Weyl module; since
it is a highest weight module of highest weight $\lambda$, it has a
unique maximal submodule with corresponding simple quotient denoted by
$L_q(\lambda)$. It can be shown by \cite{GL} that (still assuming
$\Bbbk$ is a field) the set of $\L_q(\lambda)$, for $\lambda \in \pi$,
is a complete set of simple $\Schur_q(\pi)$-modules, up to
isomorphism.

\begin{lem}
  Let $\lambda_1$ be a maximal element of $\pi$. Put $\pi' = \pi
  -\{\lambda_1\}$ and consider the quotient map $p = p_{\pi, \pi'}$.
  Let $p_\lambda$ be the restriction of $p$ to $\Schur_q(\pi)
  1_\lambda$, for each $\lambda \in W\pi$. Then $p_\lambda$ maps
  $\Schur_q(\pi) 1_\lambda$ onto $\Schur_q(\pi')1_\lambda$ or zero,
  depending whether $\lambda \in W\pi'$ or not.
\end{lem}

\begin{proof}
  The map $p$ preserves weight spaces. More precisely, if $x1_\lambda
  = x$ for some $x \in \Schur_q(\pi)$ then $p(x 1_\lambda) =
  p(x)$. This implies that $p(x) 1_\lambda = p(x)$ if $\lambda \in
  W\pi'$ or $0 = p(x)$ otherwise.
\end{proof}

\subsection{}
It is clear from the defining relation \ref{ss:defrels}(c) that any
$E_i^{(b)}$ ($i \in I$, $b > 0$) acts as zero on any vector $\epsilon
\in 1_{\lambda_1} \Schur_q^+ 1_\lambda$. Thus any $0 \ne \epsilon \in
1_{\lambda_1} \Schur_q^+ 1_\lambda$ is a maximal vector of weight
$\lambda_1$. In other words, any vector in the right ideal
$\Delta_q(\lambda_1)^* = 1_{\lambda_1} \Schur_q^+ = 1_{\lambda_1} \Schur_q$
of right weight $\lambda$ is a maximal vector of (left) weight
$\lambda_1$. Let
\[
  \epsilon_1, \dots , \epsilon_{r-1}, \epsilon_r
\]
be a $\Bbbk$-basis for the biweight space $1_{\lambda_1} \Schur_q^+
1_\lambda$. Then $\epsilon_1, \dots , \epsilon_{r-1}, \epsilon_r$ are
linearly independent maximal vectors of weight $\lambda_1$ in $\Schur_q
1_\lambda$. Hence the module $P = \Schur_q 1_\lambda$ contains the
submodule
\[
  P_1 = \textstyle \bigoplus_{j=1}^r \Schur_q \epsilon_j 1_\lambda
  \simeq \bigoplus^r \Delta_q(\lambda_1). 
\]
This is an isotypic submodule of the projective module $P = \Schur_q
1_\lambda$.  The number $r$ of direct summands in $P_1$ isomorphic to
$\Delta_q(\lambda_1)$ is precisely the rank over $\Bbbk$ of the
$\lambda$ weight space of $\Delta_q(\lambda_1)$.

\begin{lem}
  Let $\Schur_q = \Schur_q(\pi)$. Let $P = \Schur_q 1_\lambda$ for
  some $\lambda \in W\pi$. The kernel of the restriction $p_\lambda$
  of $p = p_{\pi,\pi'}$ as above to $\Schur_q 1_\lambda$ is precisely
  the isotypic submodule $P_1 = \bigoplus_{j=1}^r \Schur_q \epsilon_j
  1_\lambda \simeq \bigoplus^r \Delta_q(\lambda_1)$ generated by
  $1_{\lambda_1} \Schur_q^+ 1_\lambda$. Here $r = \rk_{\Bbbk}
  1_{\lambda_1} \Schur_q^+ 1_\lambda = \rk_{\Bbbk} 1_\lambda
  \Delta_q(\lambda_1)$.
\end{lem}

\begin{proof}
  As $\lambda$ varies over $W\pi$, by counting up the number of copies
  of $\Delta_q(\lambda_1)$ in each $\Schur_q 1_\lambda$ we get
  $\rk_{\Bbbk} \Delta_q(\lambda_1)$ of them, since
  $\Delta_q(\lambda_1)^* = 1_{\lambda_1}\Schur_q^+ =
  \bigoplus_{\lambda \in W\pi} 1_{\lambda_1}\Schur_q^+ 1_\lambda$ has
  the same rank over $\Bbbk$ as $\Delta_q(\lambda_1)$. We know the
  kernel of $p = p_{\pi,\pi'}$ contains exactly this number of copies
  of $\Delta_q(\lambda_1)$.
\end{proof}

We now obtain the main result of this section. This is a $q$-analogue
of Lemma 1.2c of \cite{Donkin:SA3}.

\begin{prop}
  Let $\lambda \in W\pi$. Let $\pi = \{ \lambda_1, \lambda_2, \dots,
  \lambda_m \}$ be ordered so that each $\lambda_j$ is a maximal
  element of $\{\lambda_j, \dots, \lambda_m\}$, for $j = 1, \dots,
  m$. Then the projective module $P = \Schur_q 1_\lambda$ has a
  filtration
  \[
  0 = P_0 \subset P_1 \subset P_2 \subset \cdots \subset P_n = P
  \]
  such that each successive quotient $P_j/ P_{j-1}$ is isotypic of
  type $\Delta_q(\lambda_j)$. The multiplicity of
  $\Delta_q(\lambda_j)$ in $P_j/ P_{j-1}$ is the rank over $\Bbbk$ of
  the weight space $1_\lambda \Delta_q(\lambda_j)$.
\end{prop}

\begin{proof}
This follows by induction. The base case is the preceding lemma.  Put
$\pi_j = \{ \lambda_j , \dots, \lambda_m \}$, and let $\Phi_j = \pi -
\pi_j$. Then $\pi_j$ is saturated and $\Phi_j$ is cosaturated, for
each $j$. Let $k \ge 2$. By induction $P_{k-1}$ satisfies the
statement (for each $\lambda \in \pi$). By considering the map
$p_{\pi, \pi_k}$ with kernel $\Schur_q[\Phi_{k-1}]$ we see from the
preceding lemma again that the result holds for $P_k$.
\end{proof}

\end{document}